\documentclass[12pt]{amsart}%
\usepackage{graphicx}
\usepackage[small,nohug,heads=littlevee]{diagrams}
\usepackage{tikz}
\usepackage{amscd}
\usepackage{geometry}
\usepackage{amsmath}
\usepackage{amsfonts}
\usepackage{amssymb}
\usepackage{enumerate}%
\providecommand{\U}[1]{\protect\rule{.1in}{.1in}}
\usetikzlibrary{matrix,arrows,decorations.pathmorphing}
\newtheorem{theorem}{Theorem}[section]
\theoremstyle{plain}

\newtheorem{corollary}[theorem]{Corollary}

\newtheorem{lemma}[theorem]{Lemma}

\newtheorem{proposition}[theorem]{Proposition}

\theoremstyle{definition}
\newtheorem{definition}[theorem]{Definition}
\newtheorem{remark}{Remark}
\newtheorem{example}{Example}
\numberwithin{equation}{section}

\begin{document}
\title[Compactifications of manifolds with boundary]{Compactifications of manifolds with boundary}
\author{Shijie Gu}
\address{Department of Mathematics\\
County College of Morris, Randolph, NJ, 07869}
\email{sgu@ccm.edu}
\author{Craig R. Guilbault}
\address{Department of Mathematical Sciences\\
University of Wisconsin-Milwaukee, Milwaukee, WI 53201}
\email{craigg@uwm.edu}
\thanks{This research was supported in part by Simons Foundation Grants 207264 and
427244, CRG}
\date{November 3, 2018}
\keywords{manifold, end, inward tame, completion, Z-compacification, Wall finiteness
obstruction, Whitehead torsion}

\begin{abstract}
This paper is concerned with compactifications of high-dimensional manifolds.
Siebenmann's iconic 1965 dissertation \cite{Sie65} provided necessary and
sufficient conditions for an open manifold $M^{m}$ ($m\geq6$) to be
compactifiable by addition of a manifold boundary. His theorem extends easily
to cases where $M^{m}$ is noncompact with compact boundary; however when
$\partial M^{m}$ is noncompact, the situation is more complicated. The goal
becomes a \textquotedblleft completion\textquotedblright\ of $M^{m}$, ie, a
compact manifold $\widehat{M}^{m}$ containing a compactum $A\subseteq\partial
M^{m}$ such that $\widehat{M}^{m}\backslash A\approx M^{m}$. Siebenmann did
some initial work on this topic, and O'Brien \cite{Obr83} extended that work
to an important special case. But, until now, a complete characterization had
yet to emerge. Here we provide such a characterization.

Our second main theorem involves $\mathcal{Z}$-compactifications. An important
open question asks whether a well-known set of conditions laid out by Chapman
and Siebenmann \cite{ChSi76} guarantee $\mathcal{Z}$-compactifiability for a
manifold $M^{m}$. We cannot answer that question, but we do show that those
conditions are satisfied if and only if $M^{m}\times\lbrack0,1]$ is
$\mathcal{Z}$-compactifiable. A key ingredient in our proof is the above
Manifold Completion Theorem---an application that partly explains our current
interest in that topic, and also illustrates the utility of the $\pi_{1}%
$-condition found in that theorem.

\end{abstract}
\maketitle

\section{Introduction}

This paper is about \textquotedblleft nice\textquotedblright%
\ compactifications of high-dimensional manifolds. The simplest of these
compactification is the addition of a boundary to an open manifold. That was
the topic of Siebenmann's famous 1965 dissertation \cite{Sie65}, the main
result of which can easily be extended to include noncompact manifolds with
compact boundaries. When $M^{m}$ has noncompact boundary, one may ask for a
compactification $\widehat{M}^{m}$ that \textquotedblleft
completes\textquotedblright\ $\partial M^{m}$. That is a more delicate
problem. Siebenmann addressed a very special case in his dissertation, before
O'Brien \cite{Obr83} characterized completable\ $n$-manifolds in the case
where $M^{m}$ and $\partial M^{m}$ are both 1-ended. Since completable
manifolds can have infinitely many (non-isolated) ends, O'Brien's theorem does
not imply a full characterization of completable $n$-manifolds. We obtain such
a characterization here, thereby completing an unfinished chapter in the study
of noncompact manifolds.

A second type of compactification considered here is the $\mathcal{Z}%
$-compactification. These are similar to the compactifications discussed
above---in fact, those are special cases---but $\mathcal{Z}$-compactifications
are more flexible. For example, a $\mathcal{Z}$-boundary for an open manifold
need not be a manifold, and a manifold that admits no completion can admit a
$\mathcal{Z}$-compactification. These compactifications have proven to be
useful in both geometric group theory and manifold topology, for example, in
attacks on the Borel and Novikov Conjectures. A major open problem (in our
minds) is a characterization of $\mathcal{Z}$-compactifiable manifolds. A set
of necessary conditions was identified by Chapman and Siebenmann
\cite{ChSi76}, and it is hoped that those conditions are sufficient. We prove
what might be viewed the next best thing: If $M^{m}$ satisfies the
Chapman-Siebenmann conditions (and $m\neq4$), then $M^{m}\times\left[
0,1\right]  $ is $\mathcal{Z}$-compactifiable. We do this by proving that
$M^{m}\times\left[  0,1\right]  $ is completable---an application that partly
explains the renewed interest in manifold completions, and also illustrates
the usefulness of the conditions found in the Manifold Completion Theorem.

\subsection{The Manifold Completion Theorem.}

An $m$-manifold $M^{m}$ with (possibly empty) boundary is \emph{completable}
if there exists a compact manifold $\widehat{M}^{m}$ and a compactum
$C\subseteq\partial\widehat{M}^{m}$ such that $\widehat{M}^{m}\backslash C$ is
homeomorphic to $M^{m}$. In this case $\widehat{M}^{m}$ is called a
\emph{(manifold) completion }of $M^{m}$. A primary goal of this paper is the
following characterization theorem for $m\geq6$. Definitions will be provided subsequently.

\begin{theorem}
[Manifold Completion Theorem]\label{Th: Completion Theorem}An $m$-manifold
$M^{m}$ ($m\geq6$) is completable if and only if

\begin{enumerate}
[(a)]

\item \label{Char2}$M^{m}$ is inward tame,

\item \label{Char1}$M^{m}$ is peripherally $\pi_{1}$-stable at infinity,

\item \label{Char3}$\sigma_{\infty}(M^{m})\in\underleftarrow{\lim}\left\{
\widetilde{K}_{0}(\pi_{1}(N))\mid N\text{ a clean neighborhood of
infinity}\right\}  $ is zero, and

\item \label{Char4}$\tau_{\infty}\left(  M^{m}\right)  \in\underleftarrow{\lim
}^{1}\left\{  \operatorname*{Wh}(\pi_{1}(N))\mid N\text{ a clean neighborhood
of infinity}\right\}  $ is zero.
\end{enumerate}
\end{theorem}

Together, Conditions (\ref{Char2}) and (\ref{Char3}) ensure that (nice)
neighborhoods of infinity have finite homotopy type, while Condition
(\ref{Char4}) allows one to upgrade certain, naturally arising, homotopy
equivalences to simple homotopy equivalences. These conditions have arisen in
other contexts, such as \cite{Sie65} and \cite{ChSi76}. 

Condition (\ref{Char1}) can be thought of as \textquotedblleft$\pi_{1}%
$-stability rel boundary\textquotedblright; it seems unique to the situation
at hand. In the special case where $M^{m}$ is 1-ended and $N_{0}\supseteq
N_{1}\supseteq\cdots$ is a cofinal sequence of (nice) connected neighborhoods
of infinity, it demands that each sequence%
\[
\pi_{1}\left(  \partial_{M}N_{i}\cup N_{i+1}\right)  \leftarrow\pi_{1}\left(
\partial_{M}N_{i}\cup N_{i+2}\right)  \leftarrow\pi_{1}\left(  \partial
_{M}N_{i}\cup N_{i+3}\right)  \leftarrow\cdots
\]
be stable where $\partial_{M}N_{i}$ denotes $\partial M^{m}\cap N_{i}$. This
reduces to ordinary $\pi_{1}$-stability when $\partial M^{m}$ is compact. A
complete discussion of this condition can be found in
\S \ref{Section: Peripheral stability}.

\begin{remark}
\label{Remark: Comments on Completion Theorem}Several comments are in order:

\begin{enumerate}
\item Dimensions $\leq5$ are discussed briefly in
\S \ref{Section: Low dimensional results}; our main focus is $m\geq6$.

\item If $\partial M^{m}$ is compact and $M^{m}$ is inward tame then $M^{m}$
has finitely many ends (see
\S \ref{Section: finite domination and inward tameness}), so the ends are
isolated and disjoint from $\partial M^{m}$. In that case Theorem
\ref{Th: Completion Theorem} reduces to Siebenmann's dissertation
\cite{Sie65}. As such, Theorem \ref{Th: Completion Theorem} can be viewed as a
generalization of \cite{Sie65}.

\item \label{Item on OBrien in Remark 1}The special case of the Manifold
Completion Theorem, where $M^{m}$ and $\partial M^{m}$ are $1$-ended, was
proved by O'Brien \cite{Obr83}; that is where \textquotedblleft peripheral
$\pi_{1}$-stability\textquotedblright\ was first defined. But since candidates
for completion can be infinite-ended (e.g., let $C\subseteq S^{m-1}$ be a
Cantor set and $M^{m}=B^{m}\backslash C)$, the general theorem is not a
corollary. In the process of generalizing \cite{Obr83}, we simplify the proof
presented there and correct an error in the formulation of Condition
(\ref{Char3}). We also exhibit some interesting examples which answer a
question posed by O'Brien about a possible weakening Condition (\ref{Char1}).

\item If Condition (\ref{Char1}) is removed from Theorem
\ref{Th: Completion Theorem}, one arrives at Chapman and Siebenmann's
conditions for characterizing $\mathcal{Z}$-compactifiable Hilbert cube
manifolds \cite{ChSi76}. A $\mathcal{Z}$-compactification theorem for
finite-dimensional manifolds is the subject of the second main result of this
paper. We will describe that theorem and the necessary definitions
now.\bigskip
\end{enumerate}
\end{remark}

\subsection{The Stable $\mathcal{Z}$-compactification Theorem for Manifolds}

To extend the idea of a completion to Hilbert cube manifolds Chapman and
Siebenmann introduced the notion of a \textquotedblleft$\mathcal{Z}%
$-compactification\textquotedblright. A compactification $\widehat{X}=X\sqcup
Z$ of a space $X$ is a $\mathcal{Z}$\emph{-compactification }if there is a
homotopy $H:\widehat{X}\times\left[  0,1\right]  \rightarrow\widehat{X}$ such
that $H_{0}=\operatorname*{id}_{\widehat{X}}$ and $H_{t}\left(  \widehat{X}%
\right)  \subseteq X$ for all $t>0$. Subsequently, this notion has been
fruitfully applied to more general spaces---notably, finite-dimensional
manifolds and complexes; see, for example, \cite{BeMe91},\cite{CaPe95}%
,\cite{FeWe95},\cite{AnGu99}, and \cite{FaLa05}. A completion of of a
finite-dimensional manifold is a $\mathcal{Z}$-compactification, but a
$\mathcal{Z}$-compactification need not be a completion. In fact, a manifold
that allows no completion can still admit a $\mathcal{Z}$-compactification;
the exotic universal covers constructed by Mike Davis are some of the most
striking examples (just apply \cite{ADG97}). Such manifolds must satisfy
Conditions (\ref{Char2}), (\ref{Char3}) and (\ref{Char4}), but the converse
remains open.\medskip

\noindent\textbf{Question. }\emph{Does every finite-dimensional manifold that
satisfies Conditions (\ref{Char2}), (\ref{Char3}) and (\ref{Char4}) of Theorem
\ref{Th: Completion Theorem} admit a }$\mathcal{Z}$\emph{-compactification?}%
\medskip

This question was posed more generally in \cite{ChSi76} for locally compact
ANRs, but in \cite{Gui01} a 2-dimensional polyhedral counterexample was
constructed. The manifold version remains open. In this paper, we prove a best
possible \textquotedblleft stabilization theorem\textquotedblright\ for manifolds.

\begin{theorem}
[Stable $\mathcal{Z}$-compactification Theorem for Manifolds]%
\label{Th: Stabilization}An $m$-manifold $M^{m}$ ($m\geq5$) satisfies
Conditions (\ref{Char2}), (\ref{Char3}) and (\ref{Char4}) of Theorem
\ref{Th: Completion Theorem}, if an only if $M^{m}\times\left[  0,1\right]  $
admits a $\mathcal{Z}$-compactification. In fact, $M^{m}\times\left[
0,1\right]  $ is completable if and only if $M^{m}$ satisfies those conditions.
\end{theorem}

\begin{remark}
In \cite{Fer00}, Ferry showed that if a locally finite $k$-dimensional
polyhedron $X$ satisfies Conditions (\ref{Char2}), (\ref{Char3}) and
(\ref{Char4}), then $X\times\left[  0,1\right]  ^{2k+5}$ is $\mathcal{Z}%
$-compactifiable. Theorem \ref{Th: Completion Theorem} can be viewed as a
sharpening of Ferry's theorem in cases where $X$ is a manifold.\bigskip
\end{remark}

\subsection{Outline of this paper}

The remainder of this paper is organized as follows. In
\S \ref{Section: Low dimensional results} we review the status of Theorem
\ref{Th: Completion Theorem} in dimensions $<6$. In
\S \ref{Section: Terminology} we fix some terminology and notation; then in
\S \ref{Section: Peripheral stability}-\ref{Section: Whitehead obstruction},
we carefully discuss each of the four conditions present in Theorem
\ref{Th: Completion Theorem}. In
\S \ref{Section: Proof of Completion Theorem: necessity}%
-\ref{Section: Proof of Completion Theorem: sufficiency} we prove Theorem
\ref{Th: Completion Theorem}, and in \S \ref{Section: Stabilization Theorem}
we prove Theorem \ref{Th: Stabilization}. In \S \ref{Section: Counterexample}
we provide a counterexample to a question posed in \cite{Obr83} about a
possible relaxation of Condition (\ref{Char1}), and in
\S \ref{Section: Equivalence of various peripheral pi1-stability conditions}
we provide the proof of a technical lemma that was postponed until the end of
the paper.

\section{Manifold completions in dimensions $<6$%
\label{Section: Low dimensional results}}

The Manifold Completion Theorem is true in dimensions $\leq3$, but much
simpler versions are possible in those dimensions. For example, Tucker
\cite{Tuc74} showed that a 3-manifold can be completed if and only if each
component of each clean neighborhood of infinity has finitely generated
fundamental group---a condition that is implied by inward tameness alone.

Since we have been unable to find the optimal $2$-dimensional completion
theorem in the literature, we take this opportunity to provide such a theorem.
If $M^{2}$ has finitely generated first homology (e.g., if $M^{2}$ is inward
tame), then by classical work (see \cite{Ker23} and \cite{Ric63})
$\operatorname*{int}(M^{2})\approx\Sigma^{2}-P$, where $\Sigma^{2}$ is a
closed surface and $P$ is a finite set of points. Therefore, $M^{2}$ contains
a compact codimension 1 submanifold $C$ such that each of the the components
$\left\{  N_{i}\right\}  _{i=1}^{k}$ of $\overline{M^{2}\backslash C\text{ }}%
$is a noncompact manifold whose frontier is a circle onto which it deformation
retracts. Complete the $N_{i}$ individually as follows:

\begin{itemize}
\item[i)] If $N_{i}$ contains no portion of $\partial M^{2}$, add a circle at
infinity; and

\item[ii)] If $N_{i}$ contains components of $\partial M^{2}$, perform the
Ker\'{e}kj\'{a}rt\'{o}-Freudenthal end-point compactification to $N_{i}$.
\end{itemize}

\noindent Classification 9.26 of \cite{CKS12}, applied to each $N_{i}$ of type
ii), ensures that the result is a manifold completion of $M^{2}$. As a
consequence, we have the following:

\begin{theorem}
A connected $2$-manifold $M^{2}$ is completable if and only if $H_{1}\left(
M^{2}\right)  $ is finitely generated; in particular, Theorem
\ref{Th: Completion Theorem} is valid when $n=2$.
\end{theorem}

In dimension 5 our proof of Theorem \ref{Th: Completion Theorem} goes through
verbatim, provided it is always possible to work in neighborhoods of infinity
with boundaries in which Freedman's 4-dimensional Disk Embedding Theorem
holds. That issue is discussed in \cite{Qui82b} and \cite[\S 11.9]{FrQu90} in
the less general setting of Siebenmann's thesis, but the issues here are the
same. In the language of \cite{FrQu90}: Theorem \ref{Th: Completion Theorem}
holds provided Condition (\ref{Char1}) is strengthened to require the
existence of arbitrarily small neighborhoods of infinity with stable
peripheral $\operatorname*{pro}$-$\pi_{1}$ groups that are \textquotedblleft
good\textquotedblright. A caveat is that, whenever \cite{Fre82} is applied,
conclusions are topological, rather than PL or smooth.

Remarkably, Siebenmann's thesis fails in dimension 4 (see \cite{Wei87} and
\cite{KwSc88}). Counterexamples to his theorem are, of course, counterexamples
to Theorem \ref{Th: Completion Theorem} as well.\medskip

As for low-dimensional versions of Theorem \ref{Th: Stabilization}: if
$m\leq3$ and $M^{m}$ satisfies Condition (\ref{Char2}) then $M^{m}$ is
completable (hence $\mathcal{Z}$-compactifiable), so $M^{m}\times\left[
0,1\right]  $ is completable and $\mathcal{Z}$-compactifiable. If $m=4$, then
$M^{4}\times\left[  0,1\right]  $ is a 5-manifold, which (see
\S \ref{Section: Stabilization Theorem}) satisfies the conditions of Theorem
\ref{Th: Completion Theorem}. Whether that leads to a completion depends on
4-dimensional issues, in particular the \textquotedblleft
goodness\textquotedblright\ of the (stable) peripheral fundamental groups of
the ends of $M^{4}\times\left[  0,1\right]  $. Those groups are determined by,
but are not the same as, the fundamental groups at the ends of $M^{4}$. If
desired, a precise group-theoretic condition can be formulated from
Proposition \ref{Prop: peripheral pi1-stability for MxR} and \cite{Gui07b}.

\section{Conventions, notation, and terminology\label{Section: Terminology}}

For convenience, all manifolds are assumed to be piecewise-linear (PL). That
assumption is particularly useful for the topic at hand, since numerous
instances of \textquotedblleft smoothing corners\textquotedblright\ would be
required in the smooth category (an issue that is covered nicely in
\cite{Obr83}). With proper attention to such details, analogous theorems can
be obtained in the smooth or topological category. Unless stated otherwise, an
$m$-manifold $M^{m}$ is permitted to have a boundary, denoted $\partial M^{m}%
$. We denote the \emph{manifold interior} by $\operatorname*{int}M^{m}$. For
$A\subseteq M^{m}$, the \emph{point-set interior} will be denoted
$\operatorname*{Int}_{M^{m}}A$ and the \emph{frontier} by $\operatorname*{Fr}%
_{M^{m}}A$ (or for conciseness, $\operatorname*{Int}_{M}A$ and the
\emph{frontier} by $\operatorname*{Fr}_{M}A$). A \emph{closed manifold} is a
compact boundaryless manifold, while an \emph{open manifold} is a non-compact
boundaryless manifold.

For $q<m$, a $q$-dimensional submanifold $Q^{q}\subseteq M^{m}$ is
\emph{properly embedded} if it is a closed subset of $M^{m}$ and $Q^{q}%
\cap\partial M^{m}=\partial Q^{q}$; it is \emph{locally flat} if each
$p\in\operatorname*{int}Q^{q}$ has a neighborhood pair homeomorphic to
$\left(
\mathbb{R}
^{m},%
\mathbb{R}
^{q}\right)  $ and each $p\in\partial Q^{q}$ has a neighborhood pair
homeomorphic to $\left(
\mathbb{R}
_{+}^{m},%
\mathbb{R}
_{+}^{q}\right)  $. By this definition, the only properly embedded codimension
$0$ submanifolds of $M^{m}$ are unions of its connected components; a more
useful type of codimension $0$ submanifold is the following: a codimension $0$
submanifold $Q^{m}\subseteq M^{m}$ is \emph{clean} if it is a closed subset of
$M^{m}$ and $\operatorname*{Fr}_{M}Q^{m}$ is a properly embedded locally flat
(hence, bicollared) $\left(  m-1\right)  $-submanifold of $M^{m}$. In that
case, $\overline{M^{m}\backslash Q^{m}}$ is also clean, and
$\operatorname*{Fr}_{M}Q^{m}$ is a clean codimension $0$ submanifold of both
$\partial Q^{m}$ and $\partial(\overline{M^{m}\backslash Q^{m}})$.

When the dimension of a manifold or submanifold is clear, we sometimes omit
the superscript; for example, denoting a clean codimension $0$ submanifold by
$Q$. Similarly, when the ambient space is clear, we denote (point-set)
interiors and frontiers by $\operatorname*{Int}A$ and $\operatorname*{Fr}A$

For any codimension $0$ clean submanifold $Q\subseteq M^{m}$, let
$\partial_{M}Q$ denote $Q\cap\partial M^{m}$; alternatively $\partial
_{M}Q=\partial Q\backslash\operatorname*{int}(\operatorname*{Fr}Q)$.
Similarly, we will let $\operatorname*{int}_{M}Q$ denote $Q\cap
\operatorname*{int}M^{m}$; alternatively $\operatorname*{int}_{M}%
Q=Q\backslash\partial M^{m}$.

\section{Ends, $\operatorname*{pro}$-$\pi_{1}$, and the peripheral $\pi_{1}%
$-stability condition\label{Section: Peripheral stability}}

\subsection{Neighborhoods of infinity, partial neighborhoods of infinity, and
ends}

Let $M^{m}$ be a connected manifold. A \emph{clean neighborhood of infinity}
in $M^{m}$ is a clean codimension $0$ submanifold $N\subseteq M^{m}$ for which
$\overline{M^{m}\backslash N}$ is compact. Equivalently, a clean neighborhood
of infinity is a set of the form $\overline{M^{m}\backslash C}$ where $C$ is a
compact clean codimension $0$ submanifold of $M^{m}$. A \emph{clean compact
exhaustion} of $M^{m}$ is a sequence $\left\{  C_{i}\right\}  _{i=1}^{\infty}$
of clean compact connected codimension $0$ submanifolds with $C_{i}%
\subseteq\operatorname*{Int}_{M}C_{i+1}$ and $\cup C_{i}=M^{m}$. By letting
$N_{i}=\overline{M^{m}\backslash C_{i}}$ we obtain the corresponding
\emph{cofinal sequence of clean neighborhoods of infinity}. Each such $N_{i}$
has finitely many components $\left\{  N_{i}^{j}\right\}  _{j=1}^{k_{i}}$. By
enlarging $C_{i}$ to include all of the compact components of $N_{i}$, we can
arrange that each $N_{i}^{j}$ is noncompact; then, by drilling out regular
neighborhoods of arcs connecting the various components of each
$\operatorname*{Fr}_{M}N_{i}^{j}$ (further enlarging $C_{i}$), we can also
arrange that each $\operatorname*{Fr}_{M}N_{i}^{j}$ is connected. A clean
$N_{i}$ with these latter two properties is called a $0$%
\emph{-neigh\-bor\-hood of infinity}. Most constructions in this paper will
begin with a clean compact exhaustion of $M^{m}$ with a corresponding cofinal
sequence of clean \smallskip$0$-neigh\-bor\-hoods of infinity.

Assuming the above arrangement, an \emph{end }$\varepsilon$ of $M^{m}$ is
determined by a nested sequence $\left(  N_{i}^{k_{i}}\right)  _{i=1}^{\infty
}$ of components of the $N_{i}$; each component is called a \emph{neighborhood
of }$\varepsilon$. More generally, any subset of $M^{m}$ that contains one of
the $N_{i}^{k_{i}}$ is a neighborhood of $\varepsilon$, and any nested
sequence $\left(  W_{j}\right)  _{j=1}^{\infty}$ of connected neighborhoods of
$\varepsilon$, for which $\cap W_{j}=\varnothing$, also determines the end
$\varepsilon$. A more thorough discussion of ends can be found in
\cite{Gui16}. Here we will abuse notation slightly by writing $\varepsilon
=\left(  N_{i}^{k_{i}}\right)  _{i=1}^{\infty}$, keeping in mind that a
sequence representing $\varepsilon$ is not unique.

At times we will have need to discuss components $\left\{  N^{j}\right\}  $ of
a neighborhood of infinity $N$ without reference to a specific end of $M^{m}$.
In that situation, we will refer to the $N^{j}$ as a \emph{partial
neighborhoods of infinity for }$M^{m}$ (\emph{partial }\smallskip
$0$\emph{-neigh\-bor\-hoods }if\emph{ }$N$ is a \smallskip$0$-neigh\-bor\-hood
of infinity). Clearly every noncompact clean connected codimension $0$
submanifold of $M^{m}$ with compact frontier is a partial neighborhood of
infinity with respect to an appropriately chosen compact $C$; if its frontier
is connected it is a partial \smallskip$0$-neigh\-bor\-hood of
infinity.\medskip

\subsection{The fundamental group of an end}

For each end $\varepsilon$ of $M^{m}$, we will define the \emph{fundamental
group at} $\varepsilon$ by using inverse sequences. Two inverse sequences of
groups $A_{0}\overset{\alpha_{1}}{\longleftarrow}A_{1}\overset{\alpha
_{2}}{\longleftarrow}A_{3}\overset{\alpha_{3}}{\longleftarrow}\cdots$ and
$B_{0}\overset{\beta_{1}}{\longleftarrow}B_{1}\overset{_{\beta_{2}%
}}{\longleftarrow}B_{3}\overset{\beta_{3}}{\longleftarrow}\cdots$ are
\emph{pro-isomorphic} if they contain subsequences that fit into a commutative
diagram of the form
\begin{equation}
\begin{diagram} G_{i_{0}} & & \lTo^{\lambda_{i_{0}+1,i_{1}}} & & G_{i_{1}} & & \lTo^{\lambda_{i_{1}+1,i_{2}}} & & G_{i_{2}} & & \lTo^{\lambda_{i_{2}+1,i_{3}}}& & G_{i_{3}}& \cdots\\ & \luTo & & \ldTo & & \luTo & & \ldTo & & \luTo & & \ldTo &\\ & & H_{j_{0}} & & \lTo^{\mu_{j_{0}+1,j_{1}}} & & H_{j_{1}} & & \lTo^{\mu_{j_{1}+1,j_{2}}}& & H_{j_{2}} & & \lTo^{\mu_{j_{2}+1,j_{3}}} & & \cdots \end{diagram} \label{basic ladder diagram}%
\end{equation}
where the connecting homomorphisms in the subsequences are (as always)
compositions of the original maps. An inverse sequence is \emph{stable }if it
is pro-isomorphic to a constant sequence $C\overset{\operatorname*{id}%
}{\longleftarrow}C\overset{\operatorname*{id}}{\longleftarrow}%
C\overset{\operatorname*{id}}{\longleftarrow}\cdots$. Clearly, an inverse
sequence is pro-isomorphic to each of its subsequences; it is stable if and
only if it contains a subsequence for which the images stabilize in the
following manner
\begin{equation}
\begin{diagram} G_{0}& & \lTo^{{\lambda}_{1}} & & G_{1} & & \lTo^{{\lambda}_{2}} & & G_{2} & & \lTo^{{\lambda}_{3}} & & G_{3} &\cdots\\ & \luTo & & \ldTo & & \luTo & & \ldTo & & \luTo & & \ldTo & \\ & & \operatorname{Im}\left( \lambda_{1}\right) & & \lTo^{\cong} & & \operatorname{Im}\left( \lambda _{2}\right) & &\lTo^{\cong} & & \operatorname{Im}\left( \lambda_{3}\right) & & \lTo^{\cong} & &\cdots & \\ \end{diagram} \label{Standard stability ladder}%
\end{equation}
where all unlabeled homomorphisms are restrictions or inclusions. (Here we
have simplified notation by relabelling the entries in the subsequence with
integer subscripts.)

Given an end $\varepsilon=\left(  N_{i}^{k_{i}}\right)  _{i=1}^{\infty}$,
choose a ray $r:[1,\infty)\rightarrow M^{m}$ such that $r\left(
[i,\infty)\right)  \subseteq N_{i}^{k_{i}}$ for each integer $i>0$ and form
the inverse sequence
\begin{equation}
\pi_{1}\left(  N_{1}^{k_{1}},r\left(  1\right)  \right)  \overset{\lambda
_{2}}{\longleftarrow}\pi_{1}\left(  N_{2}^{k_{2}},r\left(  2\right)  \right)
\overset{\lambda_{3}}{\longleftarrow}\pi_{1}\left(  N_{3}^{k_{3}},r\left(
3\right)  \right)  \overset{\lambda_{4}}{\longleftarrow}\cdots
\label{sequence: pro-pi1}%
\end{equation}
where each $\lambda_{i}$ is an inclusion induced homomorphism composed with
the change-of-basepoint isomorphism induced by the path $\left.  r\right\vert
_{\left[  i-1,i\right]  }$. We refer to $r$ as the \emph{base ray} and the
sequence (\ref{sequence: pro-pi1}) as a representative of the
\textquotedblleft fundamental group at $\varepsilon$ based at $r$%
\textquotedblright\ ---denoted $\operatorname*{pro}$-$\pi_{1}\left(
\varepsilon,r\right)  $. Any similarly obtained representation (e.g., by
choosing a different sequence of neighborhoods of $\varepsilon$) using the
same base ray can be seen to be pro-isomorphic. We say \emph{the fundamental
group at }$\varepsilon$\emph{ is stable} if (\ref{sequence: pro-pi1}) is a
stable sequence. A key observation from the theory of ends is that stability
of $\operatorname*{pro}$-$\pi_{1}(\varepsilon,r$) depends on \emph{neither
}the choice of neighborhoods nor that of the base ray. See \cite{Gui16} or
\cite{Geo08}.\medskip

\subsection{Relative connectedness, relative $\pi_{1}$-stability, and the
peripheral $\pi_{1}$-stability condition}

Let $Q$ be a manifold and $A\subseteq\partial Q$. We say that $Q$ is\emph{
}$A$-\emph{connected at infinity} if $Q$ contains arbitrarily small
neighborhoods of infinity $V$ for which $A\cup V$ is connected.

\begin{example}
If $P$ is a compact manifold with connected boundary, $X\subseteq\partial P$
is a closed set, and $Q=P\backslash X$, then $Q$ has one end for each
component of $X$ but $Q$ is $\partial Q$-connected at infinity. More
generally, if $B$ is a clean connected codimension $0$ manifold neighborhood
of $X$ in $\partial P$ and $A=B\backslash X$, then $Q$ is $A$-connected at infinity.
\end{example}

The following lemma is straightforward.

\begin{lemma}
\label{Lemma: relA connected versus Q-A 1-ended}Let $Q$ be a noncompact
manifold and $A$ a clean codimension $0$ submanifold of $\partial Q$. Then $Q$
is $A$-connected at infinity if and only if $Q\backslash A$ is $1$-ended.
\end{lemma}

If $A\subseteq\partial Q$ and $Q$ is $A$-connected at infinity: let $\left\{
V_{i}\right\}  $ be a cofinal sequence of clean neighborhoods of infinity for
which each $A\cup V_{i}$ is connected; choose a ray $r:[1,\infty
)\rightarrow\operatorname*{int}Q$ such that $r\left(  [i,\infty)\right)
\subseteq V_{i}$ for each $i>0$; and form the inverse sequence%
\begin{equation}
\pi_{1}\left(  A\cup V_{1},r\left(  1\right)  \right)  \overset{\mu
_{2}}{\longleftarrow}\pi_{1}\left(  A\cup V_{2},r\left(  2\right)  \right)
\overset{\mu_{3}}{\longleftarrow}\pi_{1}\left(  A\cup V_{3},r\left(  3\right)
\right)  \overset{\mu_{4}}{\longleftarrow}\cdots
\label{sequence: rel A pro-pi1}%
\end{equation}
where bonding homomorphisms are obtained as in (\ref{sequence: pro-pi1}). We
say $Q$ \emph{is }$A$-$\pi_{1}$\emph{-stable at infinity} if
(\ref{sequence: rel A pro-pi1}) is stable. Independence of this property from
the choices of $\left\{  V_{i}\right\}  $ and $r$ follows from the traditional
theory of ends by applying Lemmas
\ref{Lemma: relA connected versus Q-A 1-ended} and
\ref{Lemma: relA pro-pi1 versus Q-A pro-pi1}.

\begin{lemma}
\label{Lemma: relA pro-pi1 versus Q-A pro-pi1}Let $Q$ be a noncompact manifold
and $A$ a clean codimension $0$ submanifold of $\partial Q$ for which $Q$ is
$A$-connected at infinity. Then, for any cofinal sequence of clean
neighborhoods of infinity $\left\{  V_{i}\right\}  $ and ray $r:[1,\infty
)\rightarrow Q$ as described above, the sequence
(\ref{sequence: rel A pro-pi1}) is pro-isomorphic to any sequence representing
$\operatorname*{pro}$-$\pi_{1}\left(  Q\backslash A,r\right)  $.
\end{lemma}

\begin{proof}
It suffices to find a single cofinal sequence of connected neighborhoods of
infinity $\left\{  N_{i}\right\}  $ in $Q\backslash A$ for which the
corresponding representation of $\operatorname*{pro}$-$\pi_{1}\left(
Q\backslash A,r\right)  $ is pro-isomorpic to (\ref{sequence: rel A pro-pi1}).
Toward that end, for each $i$ let $C_{1}\supseteq C_{2}\supseteq\cdots$ be a
nested sequence of relative regular neighborhoods of $A$ in $Q$ such that
$\cap C_{i}=A$. By \textquotedblleft cleanness\textquotedblright\ of the
$V_{i}$, each $C_{i}$ can be chosen so that $C_{i}\cup V_{i}$ is a clean
codimension $0$ submanifold of $Q$ which deformation retracts onto $A\cup
V_{i}$. Then $N_{i}=\left(  C_{i}\cup V_{i}\right)  \backslash A$ is a clean
neighborhood of infinity in $Q\backslash A$ and $N_{i}\hookrightarrow
C_{i}\cup V_{i}$ is a homotopy equivalence. For each $i$ there is a canonical
isomorphism $\alpha_{i}:\pi_{1}\left(  A\cup V_{i},r\left(  i\right)  \right)
\rightarrow\pi_{1}\left(  N_{i},r\left(  i\right)  \right)  $ which is the
composition%
\[
\pi_{1}\left(  A\cup V_{i},r\left(  i\right)  \right)  \overset{\cong%
}{\longrightarrow}\pi_{1}\left(  C_{i}\cup V_{i},r\left(  i\right)  \right)
\overset{\cong}{\longleftarrow}\pi_{1}\left(  N_{i},r\left(  i\right)
\right)
\]
These isomorphisms fit into a commuting diagram%
\[%
\begin{array}
[c]{ccccccc}%
\pi_{1}\left(  A\cup V_{1},r\left(  1\right)  \right)  & \overset{\mu
_{2}}{\longleftarrow} & \pi_{1}\left(  A\cup V_{2},r\left(  2\right)  \right)
& \overset{\mu_{3}}{\longleftarrow} & \pi_{1}\left(  A\cup V_{3},r\left(
3\right)  \right)  & \overset{\mu_{4}}{\longleftarrow} & \cdots\\
\alpha_{1}\downarrow\cong &  & \alpha_{2}\downarrow\cong &  & \alpha
_{3}\downarrow\cong &  & \\
\pi_{1}\left(  N_{1},r\left(  1\right)  \right)  & \overset{\lambda
_{2}}{\longleftarrow} & \pi_{1}\left(  N_{2},r\left(  2\right)  \right)  &
\overset{\lambda_{3}}{\longleftarrow} & \pi_{1}\left(  N_{3},r\left(
3\right)  \right)  & \overset{\lambda_{4}}{\longleftarrow} & \cdots
\end{array}
\]
completing the proof.
\end{proof}

\begin{remark}
In the above discussion, we allow for the possibility that $A=\varnothing$. In
that case, $A$-connectedness at infinity reduces to $1$-endedness and $A$%
-$\pi_{1}$-stability to ordinary $\pi_{1}$-stability at that end.
\end{remark}

\begin{definition}
\label{Defn: peripheral local connectedness}Let $M^{m}$ be a manifold and
$\varepsilon$ an end of $M^{m}$.

\begin{enumerate}
\item \label{Condition 1: Defn of periperally locally connected}$M^{m}$ is
\emph{peripherally locally connected at infinity} if it contains arbitrarily
small \smallskip$0$-neigh\-bor\-hoods of infinity $N$ with the property that
each component $N^{j}$ is $\partial_{M}N^{j}$-connected at infinity.

\item \label{Condition 2: Defn of periperally locally connected}$M^{m}$ is
\emph{peripherally locally connected at }$\varepsilon$ if $\varepsilon$ has
arbitrarily small $0$-neigh\-bor\-hoods $P$ that are $\partial_{M}P$-connected
at infinity.
\end{enumerate}

\noindent An $N$ with the property described in condition
(\ref{Condition 1: Defn of periperally locally connected}) will be called a
\emph{strong }$0$\emph{-neigh\-bor\-hood of infinity }for $M^{m}$, and a $P$
with the property described in condition
(\ref{Condition 2: Defn of periperally locally connected}) will be called a
\emph{strong }\smallskip$0$\emph{-neigh\-bor\-hood of }$\varepsilon$. More
generally, any connected partial \smallskip$0$-neigh\-bor\-hood of infinity
$Q$ that is $\partial_{M}Q$-connected at infinity will be called a
\emph{strong partial }\smallskip$0$\emph{-neigh\-bor\-hood of infinity.}
\end{definition}

\begin{lemma}
\label{Lemma: Equivalent notions of periph. local connectedness}$M^{m}$ is
peripherally locally connected at infinity iff $M^{m}$ is peripherally locally
connected at each of its ends.
\end{lemma}

\begin{proof}
Clearly the initial condition implies the latter. For the converse, let
$N^{\prime}$ be an arbitrary neighborhood of infinity in $M^{m}$ and for each
end $\varepsilon$, let $P_{\varepsilon}$ be a \smallskip$0$-neigh\-bor\-hoods
of $\varepsilon$, contained in $N^{\prime}$, which is $\partial_{M}%
P_{\varepsilon}$-connected at infinity. By compactness of the Freudenthal
boundary of $M^{m}$, there is a finite subcollection $\left\{  P_{\varepsilon
_{k}}\right\}  _{k=1}^{n}$ that covers the end of $M^{m}$; in other words,
$C=\overline{M^{m}-\cup_{k=1}^{n}P_{\varepsilon_{k}}}$ is compact. If the
$P_{\varepsilon_{k}}$ are pairwise disjoint, we are finished; just let
$N=\cup_{k=1}^{n}P_{\varepsilon_{k}}$. If not, adjust the $P_{\varepsilon_{k}%
}$ within $N^{\prime}$ so they are in general position with respect to one
another, then let $\left\{  Q_{j}\right\}  _{j=1}^{s}$ be the set of
components of $\cup_{k=1}^{n}P_{\varepsilon_{k}}$ and note that each $Q_{j}$
is a $\partial_{M}Q_{j}$-connected partial \smallskip$0$-neigh\-bor\-hood of infinity.
\end{proof}

\begin{remark}
In the next section, we show that \emph{every} inward tame manifold $M^{m}$ is
peripherally locally connected at infinity. As a consequence, that condition
plays less prominent role than the next definition.
\end{remark}

\begin{definition}
\label{Defn: peripheral pi1-stability}Let $M^{m}$ be a manifold and
$\varepsilon$ an end of $M^{m}$.

\begin{enumerate}
\item $M^{m}$ is \emph{peripherally }$\pi_{1}$\emph{-stable at infinity} if
contains arbitrarily small strong \smallskip$0$-neigh\-bor\-hoods of infinity
$N$ with the property that each component $N^{j}$ is $\partial_{M}N^{j}$%
-$\pi_{1}$-stable at infinity.

\item $M^{m}$ is \emph{peripherally }$\pi_{1}$\emph{-stable at }$\varepsilon$
if $\varepsilon$ has arbitrarily small strong \smallskip$0$-neigh\-bor\-hoods
$P$ that are $\partial_{M}P$-$\pi_{1}$-stable at infinity.\medskip
\end{enumerate}
\end{definition}

It is easy to see that peripheral $\pi_{1}$-stability at infinity implies
peripheral $\pi_{1}$-stability at each end; and when $M^{m}$ is finite-ended,
peripheral $\pi_{1}$-stability at each end implies peripheral $\pi_{1}%
$-stability at infinity. A argument could be made for defining peripheral
$\pi_{1}$-stability at infinity to mean \textquotedblleft peripherally
$\pi_{1}$-stability at each end\textquotedblright. For us, that point is moot;
in the presence of inward tameness the two alternatives are equivalent.

\begin{lemma}
\label{Lemma: equivalence of peripheral pi1-stability conditions}An inward
tame manifold $M^{m}$ is peripherally $\pi_{1}$-stable at infinity if and only
if it is peripherally $\pi_{1}$-stable at each of its ends.
\end{lemma}

Proof of this lemma is technical, and not central to the main argument. For
that reason, we save the proof for later (see
\S \ref{Section: Equivalence of various peripheral pi1-stability conditions}).
Although it is not needed here, it would be interesting to know whether Lemma
\ref{Lemma: equivalence of peripheral pi1-stability conditions} holds without
the assumption of inward tameness.

\section{Finite domination and inward
tameness\label{Section: finite domination and inward tameness}}

A topological space $P$ is \emph{finitely dominated} if there exists a finite
polyhedron $K$ and maps $u:P\rightarrow K$ and $d:K\rightarrow P$ such that
$d\circ u\simeq\operatorname*{id}_{P}$. If choices can be made so both $d\circ
u\simeq\operatorname*{id}_{P}$ and $u\circ d\simeq\operatorname*{id}_{K}$,
i.e., $P\simeq K$, we say $P$ \emph{has finite homotopy type}. For simplicity,
we will restrict our attention to cases where $P$ is a locally finite
polyhedron---a class that contains the (PL) manifolds, submanifolds, and
subspaces considered here.

\begin{lemma}
\label{Lemma 4.1} Let $M^{m}$ be a manifold and $A\subseteq\partial M$. Then
$M^{m}$ is finitely dominated [resp., has finite homotopy type] if and only if
$M^{m}\backslash A$ is finitely dominated [resp., has finite homotopy type].
\end{lemma}

\begin{proof}
$M^{m}\backslash A\hookrightarrow M^{m}$ is a homotopy equivalence, and these
properties are homotopy invariants.
\end{proof}

\begin{lemma}
A locally finite polyhedron $P$ is finitely dominated if and only if there
exists a homotopy $H:P\times\left[  0,1\right]  \rightarrow P$ such that
$H_{0}=\operatorname*{id}_{P}$ and $\overline{H_{1}\left(  P\right)  }$ is compact.
\end{lemma}

\begin{proof}
Assuming a finite domination, as described above, the homotopy between
$\operatorname*{id}_{P}$ and $d\circ u$ has the desired property. For the
converse, let $K$ be a compact polyhedral neighborhood of $\overline
{H_{1}\left(  P\right)  }$, $u:K\hookrightarrow P$, and $d=H_{1}:P\rightarrow
K$.
\end{proof}

A locally finite polyhedron $P$ is \emph{inward tame} if it contains
arbitrarily small polyhedral neighborhoods of infinity that are finitely
dominated. Equivalently, $P$ contains a cofinal sequence $\left\{
N_{i}\right\}  $ of closed polyhedral neighborhoods of infinity each admitting
a \textquotedblleft taming homotopy\textquotedblright\ $H:N_{i}\times\left[
0,1\right]  \rightarrow N_{i}$ that pulls $N_{i}$ into a compact subset of
itself. By an application of the Homotopy Extension Property (similar to
\cite[Lemma 3.4]{GuMo18}) we can require taming homotopies to be fixed on
$\operatorname*{Fr}N_{i}$. From there, it is easy to see that, in an inward
tame polyhedron, \emph{every} closed neighborhood of infinity admits a taming
homotopy.\footnote{For a discussion of \textquotedblleft
tameness\textquotedblright\ terminology and its variants, see \cite[\S 3.5.5]%
{Gui16}.}

\begin{lemma}
\label{Lemma: inward tameness of deleted manifolds}Let $M^{m}$ be a manifold
and $A$ a clean codimension $0$ submanifold of $\partial M^{m}$. If $M^{m}$ is
inward tame then so is $M^{m}\backslash A$.
\end{lemma}

\begin{proof}
For an arbitrarily small clean neighborhood of infinity $N$ in $M^{m}$, let
$H$ be a taming homotopy that fixes $\operatorname*{Fr}N$. Then $H$ extends
via the identity to a homotopy that pulls $A\cup N$ into a compact subset of
itself, so $A\cup N$ is finitely dominated. Arguing as in Lemma
\ref{Lemma: relA pro-pi1 versus Q-A pro-pi1}, $M^{m}\backslash A$ has
arbitrarily small clean neighborhoods of infinity homotopy equivalent to such
an $A\cup N$.
\end{proof}

\begin{remark}
Important cases of Lemma \ref{Lemma: inward tameness of deleted manifolds} are
when $A=\partial M^{m}$ and when $V$ is a clean neighborhood of infinity (or a
component of one) and $A=\partial_{M}V$. Notice that Lemma
\ref{Lemma: inward tameness of deleted manifolds} is valid when $M^{m}$ is
compact and $H$ is the \textquotedblleft empty map\textquotedblright.
\end{remark}

A finitely dominated space has finitely generated homology, from which it can
be shown that an inward tame manifold with compact boundary is finite-ended
(see \cite[Prop.3.1]{GuTi03}). That conclusion fails for manifolds with
noncompact boundary; see item (\ref{Item on OBrien in Remark 1}) of Remark
\ref{Remark: Comments on Completion Theorem}. The following variation is
crucial to this paper.

\begin{proposition}
\label{Proposition: finitely many ends} If a noncompact connected manifold
$M^{m}$ and its boundary each have finitely generated homology, then $M^{m}$
has finitely many ends. More specifically, the number of ends of $M^{m}$ is
bounded above by $\dim H_{m-1}(M^{m},\partial M^{m};%
\mathbb{Z}
_{2})+1$.
\end{proposition}

\begin{proof}
Let $C$ be a clean connected compact codimension $0$ submanifold of $M^{m}$,
with the property that $N=\overline{M^{m}\backslash C}$ is a \smallskip
$0$-neigh\-bor\-hood of infinity, and let $\left\{  N^{j}\right\}  _{j=1}^{k}$
be the collection of connected components of $N^{n}$. It suffices to show that
$k\leq\dim H_{m-1}(M^{m},\partial M^{m};%
\mathbb{Z}
_{2})+1$. For the remainder of this proof (and only this proof), all homology
is with $%
\mathbb{Z}
_{2}$-coefficients.

Note that $\partial C$ is the union of clean codimension $0$ submanifolds
$\partial_{M}C$ and $\operatorname*{Fr}C$, which intersect in their common
boundary $\partial\left(  \operatorname*{Fr}C\right)  $. So by a generalized
version of Poincar\'{e} duality \cite[Th.3.43]{Hat02} and the Universal
Coefficients Theorem, for all $i$, we have
\begin{equation}
H_{i}\left(  C,\partial_{M}C\right)  \cong H_{m-i}\left(  C,\operatorname*{Fr}%
C\right)  \medskip. \label{Identity: Generalized Poincare duality}%
\end{equation}

\noindent\textsc{Claim 1. }$\dim H_{m-1}(C,\partial_{M}C)\geq k-1$.\medskip

By the long exact sequence for the pair $\left(  C,\operatorname*{Fr}C\right)
$, we have%
\[%
\begin{array}
[c]{ccccccc}%
\cdots & \rightarrow & H_{1}(C,\operatorname*{Fr}C) & \twoheadrightarrow &
\widetilde{H}_{0}(\operatorname*{Fr}C) & \rightarrow & \widetilde{H}_{0}(C)\\
&  &  &  & \shortparallel &  & \shortparallel\\
&  &  &  & (%
\mathbb{Z}
_{2})^{k-1} &  & 0
\end{array}
\]

So the claim follows from identity
(\ref{Identity: Generalized Poincare duality}).\medskip

\noindent\textsc{Claim 2. }$\operatorname*{rank}H_{m-1}(N,\partial_{M}N)\geq
k$\medskip

This claim follows from the long exact sequence for the triple $\left(
N,\partial N,\partial_{M}N\right)  $
\[%
\begin{array}
[c]{ccccccc}%
\rightarrow & H_{m}\left(  N,\partial N\right)  & \rightarrow & H_{m-1}\left(
\partial N,\partial_{M}N\right)  & \rightarrowtail & H_{m-1}\left(
N,\partial_{M}N\right)  & \rightarrow\\
& \shortparallel &  & \shortparallel &  &  & \\
& 0 &  & (%
\mathbb{Z}
_{2})^{k} &  &  &
\end{array}
\]
where triviality of $H_{m}\left(  N,\partial N\right)  $ is due to the
noncompactness of all components of $N$, and the middle equality is from
excision.\medskip

The relative Mayer-Vietoris Theorem for pairs \cite[\S 2.2]{Hat02}, applied to
$(M^{m},\partial M^{m})$ expressed as $(C\cup N,\partial_{M}C\cup\partial
_{M}N)$, contains%

\begin{equation}
H_{m-1}(\operatorname*{Fr}C,\partial\operatorname*{Fr}C)\rightarrow
H_{m-1}(C,\partial_{M}C)\oplus H_{m-1}(N,\partial_{M}N)\rightarrow
H_{m-1}(M^{m},\partial M^{m}) \label{Mayer-Vietoris}%
\end{equation}
from which we can deduce
\begin{multline*}
\dim\left(  H_{m-1}(C,\partial_{M}C)\oplus H_{m-1}(N,\partial_{M}N)\right)
\medskip\leq\\
\dim H_{m-1}(\operatorname*{Fr}C,\partial\operatorname*{Fr}C)+\dim
H_{m-1}(M^{m},\partial M^{m})
\end{multline*}
Since $H_{m-1}(\operatorname*{Fr}C,\partial\operatorname*{Fr}C)\cong(%
\mathbb{Z}
_{2})^{k}$ (from excision), then by Claims 1 and 2 we have%
\[
\left(  k-1\right)  +k\leq k+\dim H_{m-1}(M^{m},\partial M^{m}).
\]
So $k\leq\dim H_{m-1}(M^{m},\partial M^{m})+1.$
\end{proof}

\begin{corollary}
\label{Corollary: inward tame implies periph. loc. connected}If $M^{m}$ is
inward tame, then $M^{m}$ is peripherally locally connected at infinity.
\end{corollary}

\begin{proof}
By Lemma \ref{Lemma: relA connected versus Q-A 1-ended}, it suffices to show
that each compact codimension $0$ clean submanifold $D\subseteq M^{m}$ is
contained in a compact codimension $0$ clean submanifold $C\subseteq M^{m}$ so
that if $N=\overline{M^{m}\backslash C}$, then each component $N^{j}$ of $N$
has the property that $N^{j}\setminus\partial M^{m}$ is $1$-ended.

Since $M^{m}$ is inward tame, each of its clean neighborhoods of infinity is
finitely dominated, so $\overline{M^{m}\backslash D}$ has finitely many
components, each of which is finitely dominated. Let $P^{l}$ be one of those
components. Then, $\operatorname*{Fr}P^{l}$ is a compact clean codimension $0$
submanifold of $\partial D$, whose interior is the boundary of $P^{l}%
\setminus\partial M^{m}$. Since $\operatorname*{int}\left(  \operatorname*{Fr}%
P^{l}\right)  $ and $P^{l}\setminus\partial M^{m}$ each have finitely
generated homology ($P^{l}\setminus\partial M^{m}$ is finitely dominated),
then by Proposition \ref{Proposition: finitely many ends}, $P^{l}%
\setminus\partial M^{m}$ has finitely many ends. Choose a compact clean
codimension $0$ submanifold $K_{l}$ of $P^{l}\setminus\partial M^{m}$ that
intersects $\operatorname*{int}(\operatorname*{Fr}P^{l})$ nontrivially and has
exactly one (unbounded) complementary component in $P^{l}\setminus\partial
M^{m}$ for each of those ends. After doing this for each of the component
$P^{l}$ of $\overline{M^{m}\backslash D}$, let $C=D\cup\left(  \cup
K_{l}\right)  $.
\end{proof}

\section{Finite homotopy type and the $\sigma_{\infty}$%
-obstruction\label{Section: Finiteness obstruction}}

Finitely generated projective left $\Lambda$-modules $S$ and $T$ are
\emph{stably equivalent} if there exist finitely generated free $\Lambda
$-modules $F_{1\text{ }}$ and $F_{2}$ such that $S\oplus F_{1}\cong T\oplus
F_{2}$. Under the operation of direct sum, the stable equivalence classes of
finitely generated projective modules form a group $\widetilde{K}_{0}\left(
\Lambda\right)  $, the \emph{reduced projective class group} of $\Lambda$. In
\cite{Wal65}, Wall associated to each path connected finitely dominated space
$P$ a well-defined $\sigma\left(  P\right)  \in\widetilde{K}_{0}\left(
\mathbb{Z}
\lbrack\pi_{1}\left(  P\right)  ]\right)  $ which is trivial if and only if
$P$ has finite homotopy type. (Here $%
\mathbb{Z}
\lbrack\pi_{1}\left(  P\right)  ]$ denotes the integral group ring
corresponding to $\pi_{1}\left(  P\right)  $. In the literature,
$\widetilde{K}_{0}\left(
\mathbb{Z}
\lbrack G]\right)  $ is sometimes abbreviated to $\widetilde{K}_{0}\left(
G\right)  $.) As one of the necessary and sufficient conditions for
completability of a $1$-ended inward tame open manifold $M^{m}$ ($m>5$) with
stable $\operatorname*{pro}$-$\pi_{1}$, Siebenmann defined the \emph{end
obstruction} $\sigma_{\infty}\left(  M^{m}\right)  $, to be (up to sign) the
finiteness obstruction $\sigma\left(  N\right)  $ of an arbitrary clean
neighborhood of infinity $N$ whose fundamental group \textquotedblleft
matches\textquotedblright\ the stable $\operatorname*{pro}$-$\pi_{1}\left(
\varepsilon\left(  M^{m}\right)  \right)  $.\footnote{The main theorem of
\cite{Obr83} incorrectly uses $\sigma(M^{m})$ ---the finiteness obstruction of
the entire manifold $M^{m}$ --- in place of $\sigma_{\infty}\left(
M^{m}\right)  $. The mistake is an erroneous application of Siebenmann's Sum
Theorem to conclude that triviality of $\sigma(M^{m})$ implies triviality of
$\sigma\left(  N\right)  $ for each clean neighborhood of infinity $N$.
Siebenmann \cite{Sie65} (correctly) used the Sum Theorem to show that, in the
case of \emph{stable} pro-$\pi_{1}$, it is enough to check the obstruction
once---for a well-chosen clean neighborhood of infinity. He denoted that
obstruction $\sigma\left(  \varepsilon\right)  $. In our situation (and
O'Brien's) such a simplification is not possible. We use the subscripted
\textquotedblleft$\infty$\textquotedblright\ to help distinguish the general
situation from Siebenmann's special case. }

In cases where $M^{m}$ is multi-ended or has non-stable $\operatorname*{pro}%
$-$\pi_{1}$ (or both), a more general definition of $\sigma_{\infty}\left(
M^{m}\right)  $, introduced in \cite{ChSi76}, is required. Its definition
employs several ideas from \cite[\S 6]{Sie65}. First note that there is a
covariant functor $\widetilde{K}_{0}$ from groups to abelian groups taking $G$
to $\widetilde{K}_{0}(%
\mathbb{Z}
\lbrack G])$, which may be composed with the $\pi_{1}$-functor to get a
functor from path connected spaces to abelian groups; here we use an
observation by Siebenmann allowing base points to be ignored. Next extend the
functor and the finiteness obstruction to non-path-connected $P$ (abusing
notation slightly) by letting
\[
\widetilde{K}_{0}(%
\mathbb{Z}
\left[  \pi_{1}\left(  P\right)  \right]  )=%
{\textstyle\bigoplus}
\widetilde{K}_{0}(%
\mathbb{Z}
\left[  \pi_{1}\left(  P^{j}\right)  \right]  )
\]
where $\left\{  P^{j}\right\}  $ is the set of path components of $P$, and
letting
\[
\sigma\left(  P\right)  =\left(  \sigma(P^{1}),\cdots,\sigma\left(
P^{k}\right)  \right)
\]
recalling that $P$ is finitely dominated and, hence, has finitely many
components---each finitely dominated.

Now, for an inward tame locally finite polyhedron $P$ (or more generally
locally compact ANR), let $\left\{  \,N_{j}\right\}  $ be a nested cofinal
sequence of closed polyhedral neighborhoods of infinity and define%
\[
\sigma_{\infty}\left(  P\right)  =\left(  \sigma\left(  N_{1}\right)
,\sigma\left(  N_{2}\right)  ,\sigma\left(  N_{3}\right)  ,\cdots\right)
\in\underleftarrow{\lim}\left\{  \widetilde{K}_{0}[%
\mathbb{Z}
\lbrack\pi_{1}(N_{i})]\right\}
\]
The bonding maps of the target inverse sequence%
\[
\widetilde{K}_{0}[%
\mathbb{Z}
\lbrack\pi_{1}(N_{1})]\leftarrow\widetilde{K}_{0}[%
\mathbb{Z}
\lbrack\pi_{1}(N_{2})]\leftarrow\widetilde{K}_{0}[%
\mathbb{Z}
\lbrack\pi_{1}(N_{3})]\leftarrow\cdots
\]
are induced by inclusion, with the Sum Theorem for finiteness obstructions
\cite[Th.6.5]{Sie65} assuring consistency. Clearly, $\sigma_{\infty}\left(
P\right)  $ vanishes if and only if each $N_{i}$ has finite homotopy type; by
another application of the Sum Theorem, this happens if and only if
\emph{every} closed polyhedral neighborhood of infinity has finite homotopy type.

\begin{remark}
Alternatively, we could define $\sigma_{\infty}\left(  P\right)  $ to lie in
the inverse limit of the \emph{inverse system} corresponding to all closed
polyhedral neighborhoods of infinity, partially ordered by inclusion. These
inverse limits are isomorphic, and in either case, the combination of
Conditions (\ref{Char2}) and (\ref{Char3}) of Theorem
\ref{Th: Completion Theorem} is equivalent to the requirement that all clean
neighborhoods of infinity have finite homotopy type---a property referred to
as absolute inward tameness in \cite{Gui16}.
\end{remark}

We close this section with an observation that builds upon Lemma
\ref{Lemma: inward tameness of deleted manifolds}. Both play key roles in the
proof of Theorem \ref{Th: Completion Theorem}.

\begin{lemma}
\label{Lemma: absolute inward tameness of deleted manifolds}Let $M^{m}$ be a
manifold and $A$ a clean codimension $0$ submanifold of $\partial M^{m}$. If
$M^{m}$ is inward tame and $\sigma_{\infty}\left(  M^{m}\right)  $ vanishes,
then $M^{m}\backslash A$ is inward tame and $\sigma_{\infty}\left(
M^{m}\backslash A\right)  $ also vanishes.
\end{lemma}

\begin{proof}
Lemma \ref{Lemma: inward tameness of deleted manifolds} assures us that if
$M^{m}$ is inward tame, then so too is $M^{m}\backslash A$. The latter ensures
that $\sigma_{\infty}\left(  M^{m}\backslash A\right)  $ is defined. Arguing
as we did in the proof of Lemma
\ref{Lemma: inward tameness of deleted manifolds}, $M^{m}\backslash A$
contains arbitrarily small neighborhoods of infinity which are homotopy
equivalent to $A\cup N$, where $N$ is a clean neighborhood of infinity in
$M^{m}$. If $\sigma_{\infty}\left(  M^{m}\right)  =0$, then $N$ has finite
homotopy type; and since $A\cup N=\overline{A\backslash N}\cup N$, where
$\overline{A\backslash N}$ is a compact $\left(  m-1\right)  $-manifold, then
$A\cup N$ has finite homotopy type (by a direct argument or easy application
of the Sum Theorem for the finiteness obstruction). The vanishing of
$\sigma_{\infty}\left(  M^{m}\backslash A\right)  $ then follows from the
above discussion.
\end{proof}

\section{The $\tau_{\infty}$-obstruction\label{Section: Whitehead obstruction}%
}

The $\tau_{\infty}$ obstruction in Condition (\ref{Char4}) of Theorem
\ref{Th: Completion Theorem} was first defined in \cite{ChSi76} and applied to
Hilbert cube manifolds; the role it plays here is similar. It lies in the
derived limit of an inverse sequence of Whitehead groups. For a more detailed
discussion, the reader should see \cite{ChSi76}.

The \emph{derived limit} of an inverse sequence
\[
G_{0}\overset{\lambda_{1}}{\longleftarrow}G_{1}\overset{\lambda_{2}%
}{\longleftarrow}G_{2}\overset{\lambda_{3}}{\longleftarrow}\cdots
\]
of abelian groups is the quotient group:%
\[
\underleftarrow{\lim}^{1}\left\{  G_{i},\lambda_{i}\right\}  =\left(
\prod\limits_{i=0}^{\infty}G_{i}\right)  /\left\{  \left.  \left(
g_{0}-\lambda_{1}g_{1},g_{1}-\lambda_{2}g_{2},g_{2}-\lambda_{3}g_{3}%
,\cdots\right)  \right\vert \ g_{i}\in G_{i}\right\}
\]

\noindent It is a standard fact that pro-isomorphic inverse sequences of
abelian groups have isomorphic derived limits.

Suppose a manifold $M^{m}$ contains a cofinal sequence $\left\{
N_{i}\right\}  $ of clean neighborhoods of infinity with the property that
each inclusion $\operatorname*{Fr}N_{i}\hookrightarrow N_{i}$ is a homotopy
equivalence\footnote{A manifold admitting such sequence of neighborhoods of
infinity is called \emph{pseudo-collarable}. See \cite{Gui00}, \cite{GuTi03}
and \cite{GuTi06} for discussion of that topic.}. Let $W_{i}=\overline
{N_{i}\backslash N_{i+1}}$ and note that $\operatorname*{Fr}N_{i}%
\hookrightarrow W_{i}$ is a homotopy equivalence. See Figure \ref{Figure1}.

\begin{figure}[ptb]
\centering
\includegraphics[
height=3.375in,
width=3.875in
]{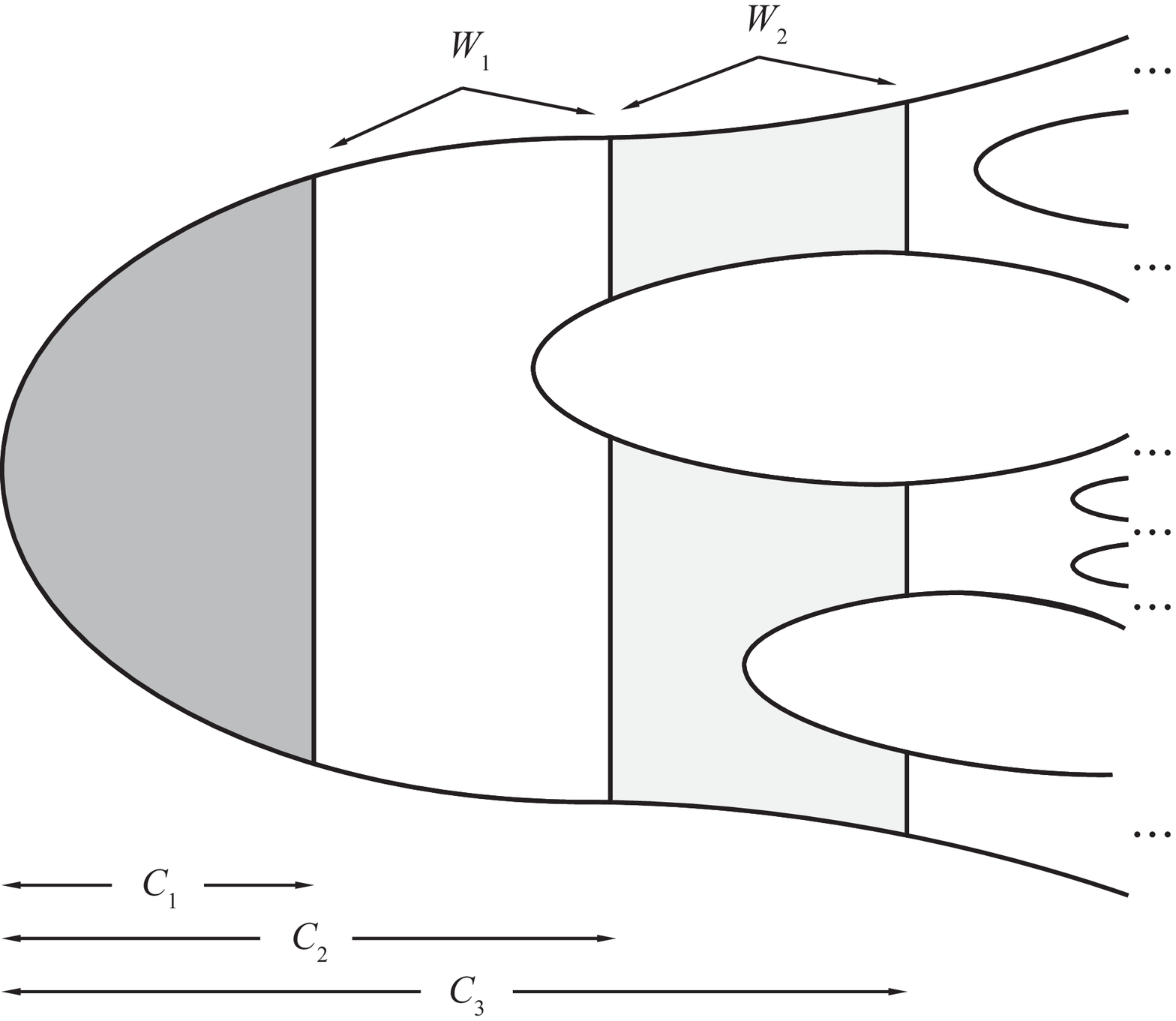}\caption{Decomposition of $M^{m}$ into $\left\{
W_{i}\right\}  _{i=1}^{\infty}$.}%
\label{Figure1}%
\end{figure}

Since $\operatorname*{Fr}N_{i}$ and $W_{i}$ are finite polyhedra, the
inclusion determines a Whitehead torsion $\tau\left(  W_{i},\operatorname*{Fr}%
N_{i}\right)  \in\operatorname*{Wh}(\pi_{1}(\operatorname*{Fr}N_{i}))$ (see
\cite{Coh73}). As in the previous section, we must allow for non-connected
$\operatorname*{Fr}N_{i}$ so we define
\[
\operatorname*{Wh}(\pi_{1}(\operatorname*{Fr}N_{i}))=%
{\textstyle\bigoplus}
\operatorname*{Wh}(\pi_{1}(\operatorname*{Fr}N_{i}^{j}))
\]
where $\left\{  \operatorname*{Fr}N_{i}^{j}\right\}  $ is the (finite) set of
components of $\operatorname*{Fr}N_{i}$ and
\[
\tau\left(  W_{i},\operatorname*{Fr}N_{i}\right)  =\left(  \tau\left(
W_{i}^{1},\operatorname*{Fr}N_{i}^{1}\right)  ,\cdots,\tau\left(  W_{i}%
^{k},\operatorname*{Fr}N_{i}^{k}\right)  \right)  .
\]
These groups fit into and inverse sequence of abelian groups
\[
\operatorname*{Wh}(\pi_{1}(N_{1}))\leftarrow\operatorname*{Wh}(\pi_{1}%
(N_{2}))\leftarrow\operatorname*{Wh}(\pi_{1}(N_{3}))\leftarrow\cdots
\]
where the bonding homomorphisms are induced by inclusions. (To match
\cite{ChSi76}, we have substituted $\pi_{1}(N_{i})$ for the canonically
equivalent $\pi_{1}(\operatorname*{Fr}N_{i})$.) Let $\tau_{i}=\tau\left(
W_{i},\operatorname*{Fr}N_{i}\right)  \in\operatorname*{Wh}(\pi_{1}(N_{i}))$.
Then%
\[
\tau_{\infty}\left(  M^{m}\right)  =[\left(  \tau_{1},\tau_{2},\tau_{3}%
,\cdots\right)  ]\in\underleftarrow{\lim}^{1}\left\{  \operatorname*{Wh}%
(\pi_{1}(N_{i}))\right\}
\]
where $[\left(  \tau_{1},\tau_{2},\tau_{3},\cdots\right)  ]$ is the coset
containing $\left(  \tau_{1},\tau_{2},\tau_{3},\cdots\right)  $.

If $\tau_{\infty}\left(  M^{m}\right)  $ is trivial, it is possible to adjust
the choices of the $N_{i}$ so that each inclusion $\operatorname*{Fr}%
N_{i}\hookrightarrow W_{i}$ has trivial torsion, and hence is a simple
homotopy equivalence. Roughly speaking, the adjustment involves
\textquotedblleft lending and borrowing torsion to and from immediate
neighbors of the $W_{i}$\textquotedblright. The procedure is as described in
\cite[\S 6]{ChSi76}, except that a Splitting Theorem for finite-dimensional
manifolds (see \cite[p.318]{Obr83}) replaces \cite[Lemma 6.1]{ChSi76}. The
reader is warned that the procedure described in \cite[\S 4]{Obr83} is flawed;
we recommend \cite{ChSi76}.

\section{Geometric characterization of completable manifolds and a review of
h- and s-cobordisms}

The following geometric characterization of completable manifolds, which has
analogs in \cite{Tuc74} and \cite{Obr83}, paves the way for the proof of
Theorem \ref{Th: Completion Theorem}. It leads naturally to the consideration
of h- and s-cobordisms, which we will briefly review for later use.

\begin{lemma}
[Geometric characterization of completable manifolds]%
\label{Lemma: geometric characterization}A non-compact manifold with boundary
$M^{m}$ is completable iff $M^{m}=\cup_{i=1}^{\infty}C_{i}$ where, for all $i$:

\begin{enumerate}
[(i)]

\item \label{one} $C_{i}$ is a compact clean codimension $0$ submanifold of
$M^{m}$,

\item \label{two} $C_{i}\subset\operatorname*{Int}C_{i+1}$, and

\item \label{three} if $W_{i}$ denotes $\overline{C_{i+1}\setminus C_{i}}$,
then $(W_{i},\operatorname*{Fr}C_{i})\approx(\operatorname*{Fr}C_{i}%
\times\left[  0,1\right]  ,\operatorname*{Fr}C_{i}\times\left\{  0\right\}  )$.
\end{enumerate}
\end{lemma}

\begin{proof}
For the forward implication, suppose $\widehat{M}^{m}$ is a compact manifold,
$A$ is closed subset set of $\partial\widehat{M}^{m}$, and $M^{m}%
=\widehat{M}^{m}\setminus A$. Write $A$ as $\cap_{i}F_{i}$, where $\left\{
F_{i}\right\}  _{i=1}^{\infty}$ is a sequence of compact clean codimension $0$
submanifolds of $\partial\widehat{M}^{m}$ with $F_{i+1}\subseteq
\operatorname{Int}F_{i}$. Let $c:\partial\widehat{M}^{m}\times\lbrack
0,1]\rightarrow\widehat{M}^{m}$ be a collar on $\partial\widehat{M}^{m}$ with
$c\left(  \partial\widehat{M}^{m}\times\left\{  0\right\}  \right)
=\partial\widehat{M}^{m}$ and, for each $i$, let $C_{i}=\widehat{M}%
^{m}\setminus c\left(  \operatorname{Int}(F_{i})\times\lbrack0,1/i)\right)  $.
Assertions (\ref{one}) and (\ref{two}) are clear. Moreover,
\begin{align*}
W_{i}  &  \approx F_{i}\times\lbrack0,1/i]\setminus\left(  \operatorname{Int}%
F_{i+1}\times\lbrack0,1/(i+1))\right) \\
&  \approx F_{i}\times\lbrack0,1/i]
\end{align*}
via a homeomorphism taking $c\left(  F_{i}\times\left\{  1/i\right\}  \right)
$ onto $F_{i}\times\left\{  1/i\right\}  $. Then, since $\operatorname*{Fr}%
C_{i}=c\left(  F_{i}\times\left\{  1/i\right\}  \cup\partial F_{i}%
\times\left[  0,1/i\right]  \right)  \approx F_{i}$, an application of
relative regular neighborhood theory allows an adjustment of that
homeomorphism so that $\operatorname*{Fr}C_{i}$ is taken onto $F_{i}%
\times\left\{  1/i\right\}  $. A reparametrization of the closed interval
completes the proof of assertion (\ref{three}). (Note that this works even
when the $F_{i}$ have multiple and varying numbers of components. See Figure
\ref{Figure2}.)

\begin{figure}[ptb]
\centering
\includegraphics[
height=2.875in,
width=4in
]{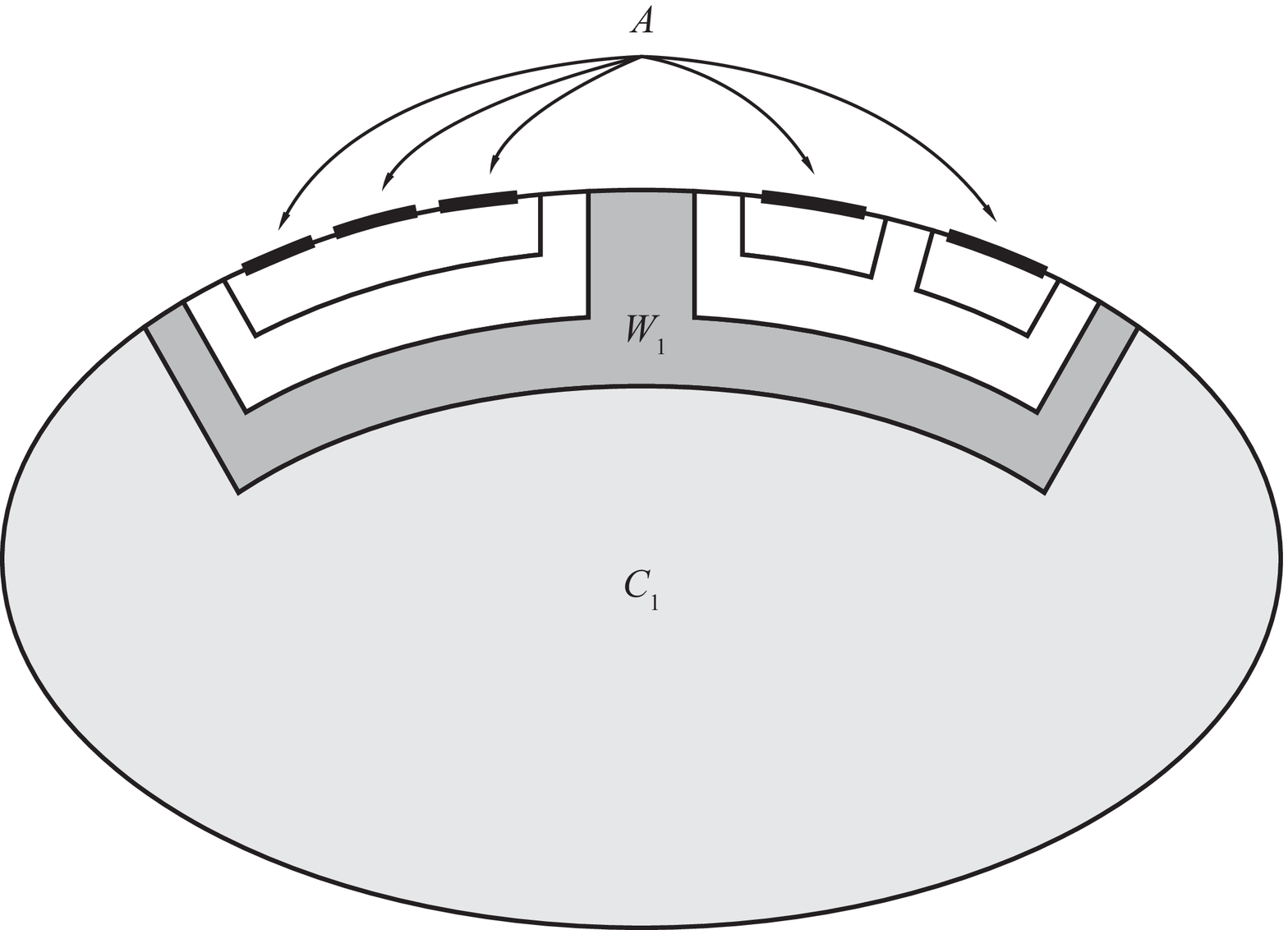}\caption{Decomposing completed $M^{m}$ into product
cobordisms.}%
\label{Figure2}%
\end{figure}

For the converse, we reverse the above procedure to embed $M^{m}$ in a copy of
$C_{1}$. Details can be found in \cite[Lemma 1]{Tuc74}.
\end{proof}

The above lemma shows that a strategy for completing a manifold is to fill up
a neighborhood of infinity in $M^{m}$ with a sequence of cobordisms, then
modify those cobordisms (when possible) so they become products.

Recall that an (absolute) \emph{cobordism} is a triple $\left(  W,A,B\right)
$, where $W$ is a manifold with boundary and $A$ and $B$ are disjoint
manifolds without boundary for which $A\cup B=\partial W$. The triple $\left(
W,A,B\right)  $ is a \emph{relative cobordism} if $A$ and $B$ are disjoint
codimension $0$ clean submanifolds of $\partial W$. In that case, there is an
associated absolute cobordism $\left(  V,\partial A,\partial B\right)  $ where
$V=\partial W\backslash\left(  \operatorname*{int}A\cup\operatorname*{int}%
B\right)  $. We view absolute cobordisms as special cases of relative
cobordisms where $V=\varnothing$. A relative cobordism is an
\emph{h-cobordism} if each of the inclusions $A\hookrightarrow W$,
$B\hookrightarrow W$, $\partial A\hookrightarrow V$, and $\partial
B\hookrightarrow V$ is a homotopy equivalence; it is an \emph{s-cobordism} if
each of these inclusions is a simple homotopy equivalence. (For convenience,
$\varnothing\hookrightarrow\varnothing$ is considered a simple homotopy
equivalence.) A relative cobordism is \emph{nice} if it is absolute or if
$\left(  V,\partial A,\partial B\right)  \approx\left(  \partial
A\times\left[  0,1\right]  ,\partial A\times\left\{  0\right\}  ,\partial
A\times\left\{  1\right\}  \right)  $. The crucial result, proof (and
additional discussion) of which may be found in \cite{RoSa82} , is the following.

\begin{theorem}
[Relative s-cobordism Theorem]A compact nice relative cobordism $\left(
W,A,B\right)  $ with $\dim W\geq6$ is a product, i.e., $\left(  W,A,B\right)
\approx\left(  A\times\left[  0,1\right]  ,A\times\left\{  0\right\}
,A\times\left\{  1\right\}  \right)  $, if and only if it is an s-cobordism.
\end{theorem}

\begin{remark}
\label{Remark: adjusting to nice relative cobordisms}A situation similar to a
nice relative cobordism occurs when $\partial W=A\cup B^{\prime}$, where $A$
and $B^{\prime}$ are codimension $0$ clean submanifolds of $\partial W$ with a
common nonempty boundary $\partial A=\partial B^{\prime}$. By choosing a clean
codimension $0$ submanifold $B\subseteq B^{\prime}$ with the property that
$B^{\prime}\backslash\operatorname*{Int}B\approx\partial B\times\left[
0,1\right]  $ we arrive at a nice relative cobordism $\left(  W,A,B\right)  $.
When this procedure is applied, we will refer to $\left(  W,A,B\right)  $ as a
\emph{corresponding }nice relative cobordism. For notational consistency, we
will always adjust the term $B^{\prime}$ on the far right of the triple
$\left(  W,A,B^{\prime}\right)  $, leaving $A$ alone.
\end{remark}

For our purposes, the following lemma will be crucial.

\begin{lemma}
\label{Lemma 7.2} Let $W$ be a compact manifold with $\partial W=A\cup
B^{\prime}$, where $A$ and $B^{\prime}$ are codimension $0$ clean submanifolds
of $\partial W$ with a common boundary. Suppose $A\hookrightarrow W$ is a
homotopy equivalence and that there is a homotopy $J:W\times\lbrack
0,1]\rightarrow W$ such that $J_{0}=\operatorname*{id}_{W}$, $J$ is fixed on
$\partial B^{\prime}$, and $J_{1}\left(  W\right)  \subseteq B^{\prime}$. Then
$B^{\prime}\hookrightarrow W$ is a homotopy equivalence, so the corresponding
nice relative cobordism $\left(  W,A,B\right)  $ is an h-cobordism.
\end{lemma}

\begin{proof}
Choose $p\in\partial A=\partial B^{\prime}$, to be used as the basepoint for
$A$, $B^{\prime}$ and $W$. Let $i:A\hookrightarrow W$ and $\iota:B^{\prime
}\hookrightarrow W$ denote inclusions and define $f:A\rightarrow B^{\prime}$
by $f\left(  x\right)  =J_{1}\left(  x\right)  $. Then
\begin{equation}
\iota\circ f=J_{1}\circ i \label{Function compositions}%
\end{equation}
Clearly $J_{1}:W\rightarrow W$ induces the identity isomorphism on $\pi
_{1}\left(  W,p\right)  $, and since $i$ is a homotopy equivalence, it induces
a $\pi_{1}$-isomorphism. So, from (\ref{Function compositions}), we may deduce
that $f_{\ast}:\pi_{1}\left(  A,p\right)  \rightarrow\pi_{1}\left(  B^{\prime
},p\right)  $ is injective. Moreover, since $f$ restricts to the identity
function mapping $\partial A$ onto $\partial B^{\prime}$, \cite{Eps66} allows
us to conclude that $f_{\ast}$ is an isomorphism. From there it follows that
$\iota_{\ast}:\pi_{1}\left(  B^{\prime},p\right)  \rightarrow\pi_{1}\left(
W,p\right)  $ is also an isomorphism.

Let $p:\widetilde{W}\rightarrow W$ be the universal covering projection,
$\widetilde{A}=p^{-1}(A)$, and $\widetilde{B}^{\prime}=p^{-1}(B^{\prime})$.
Since $i_{\ast}$ and $\iota_{\ast}$ are both $\pi_{1}$-isomorphisms these are
the universal covers of $A$ and $B^{\prime}$, respectively. By generalized
Poincar\'{e} duality for non-compact manifolds,%

\[
H_{k}(\widetilde{W},\widetilde{B}^{\prime};%
\mathbb{Z}
)\cong H_{c}^{n-k}(\widetilde{W},\widetilde{A};%
\mathbb{Z}
),
\]
where cohomology is with compact supports. Since $\widetilde{A}\hookrightarrow
\widetilde{W}$ is a proper homotopy equivalence, all of these relative
cohomology groups vanish, so $H_{k}(\widetilde{W},\widetilde{B}^{\prime};%
\mathbb{Z}
)=0$ for all $k$. By the relative Hurewicz theorem, $\pi_{k}(\widetilde{W}%
,\widetilde{B}^{\prime})=0$ for all $k$, so the same is true for $\pi
_{k}(W,B^{\prime})$. An application of Whitehead's theorem allows us to
conclude that $B^{\prime}\hookrightarrow W$ is a homotopy equivalence.
\end{proof}

\section{Proof of the Manifold Completion Theorem:
necessity\label{Section: Proof of Completion Theorem: necessity}}

We will prove necessity of the conditions in Theorem
\ref{Th: Completion Theorem} by a straightforward application of Lemma
\ref{Lemma: geometric characterization}.

\begin{proof}
[Proof of Theorem \ref{Th: Completion Theorem} (necessity)]Suppose
$\widehat{M}^{m}$ is a compact manifold and $A$ is closed subset set of
$\partial\widehat{M}^{m}$such that $M^{m}=\widehat{M}^{m}\setminus A$. As in
the proof of Lemma \ref{Lemma: geometric characterization} write $A=\cap
_{i}F_{i}$, where $\left\{  F_{i}\right\}  $ is a sequence of compact clean
codimension $0$ submanifolds of $\partial\widehat{M}^{m}$ with $F_{i+1}%
\subseteq\operatorname{Int}F_{i}$, and let $c:\partial\widehat{M}^{m}%
\times\lbrack0,1]\rightarrow\widehat{M}^{m}$ be a collar on $\partial
\widehat{M}^{m}$ with $c\left(  \partial\widehat{M}^{m}\times\left\{
0\right\}  \right)  =\partial\widehat{M}^{m}$. For each $i$, let
$\widehat{N}_{i}=c\left(  F_{i}\times\left[  0,1/i\right]  \right)  $ and
$N_{i}=\widehat{N}_{i}\backslash A$. Then $\left\{  N_{i}\right\}  $ is
cofinal sequence of clean neighborhoods of infinity in $M^{m}$ with
$\operatorname*{Fr}N_{i}=c\left(  F_{i}\times\left\{  1/i\right\}
\cup\partial F_{i}\times\left[  0,1/i\right]  \right)  $. Since $F_{i}%
\times\left\{  1/i\right\}  \cup\partial F_{i}\times\left[  0,1/i\right]
\hookrightarrow F_{i}\times\left[  0,1/i\right]  $ and $N_{i}\hookrightarrow
\widehat{N}_{i}$ are both homotopy equivalences, then so is
$\operatorname*{Fr}N_{i}\hookrightarrow N_{i}$; and since each $N_{i}$ has
finite homotopy type, conditions (\ref{Char2}) and (\ref{Char3}) of Theorem
\ref{Th: Completion Theorem} both hold (by the discussion in
\S \ref{Section: finite domination and inward tameness} and
\ref{Section: Finiteness obstruction}).

If we let $W_{i}=\overline{N_{i}\backslash N_{i+1}}$, then $\tau_{\infty
}\left(  M^{m}\right)  $ is determined by the Whitehead torsions of inclusions
$\operatorname*{Fr}N_{i}\hookrightarrow W_{i}$ (see
\S \ref{Section: Whitehead obstruction}). Associate $W_{i}$ with $F_{i}%
\times\lbrack0,1/i]$ and $\operatorname*{Fr}N_{i}$ with $F_{i}\times\left\{
1/i\right\}  \cup\partial F_{i}\times\left[  0,1/i\right]  $, as in the proof
of Lemma \ref{Lemma: geometric characterization}. Then, the fact that both
$F_{i}\times\left\{  1/i\right\}  \hookrightarrow F_{i}\times\left[
0,1/i\right]  $ and $F_{i}\times\left\{  1/i\right\}  \hookrightarrow
F_{i}\times\left\{  1/i\right\}  \cup\partial F_{i}\times\left[  0,1/i\right]
$ are simple homotopy equivalences ensures that $\tau\left(  W_{i}%
,\operatorname*{Fr}N_{i}\right)  =0$. So condition (\ref{Char4}) is satisfied.

It remains to verify the peripheral $\pi_{1}$-stability condition. Fix
$i\geq1$ and let $F_{i}^{j}$ be one component of $F_{i}$, $\widehat{N}_{i}%
^{j}=c\left(  F_{i}^{j}\times\left[  0,1/i\right]  \right)  $ and $N_{i}%
^{j}=\widehat{N}_{i}^{j}\backslash A$. Then $\partial_{M}N_{i}^{j}=c\left(
F_{i}\times\left\{  0\right\}  \right)  \backslash A$ and $N_{i}^{j}$ is
clearly $\partial_{M}N_{i}^{j}$-connected at infinity. For each $k>i$, let
$F_{k}^{\prime}$ be the union of all components of $F_{k}$ contained in
$F_{i}^{j}$, $\widehat{N_{k}^{\prime}}=c\left(  F_{k}^{\prime}\times\left[
0,1/k\right]  \right)  $ and $N_{k}^{\prime}=\widehat{N_{k}^{\prime}%
}\backslash A$. By definition, we may consider the sequence
\begin{equation}
\pi_{1}\left(  \partial_{M}N_{i}^{j}\cup N_{i+1}^{\prime}\right)
\overset{\mu_{2}}{\longleftarrow}\pi_{1}\left(  \partial_{M}N_{i}^{j}\cup
N_{i+2}^{\prime}\right)  \overset{\mu_{3}}{\longleftarrow}\pi_{1}\left(
\partial_{M}N_{i}^{j}\cup N_{i+3}^{\prime}\right)  \overset{\mu_{4}%
}{\longleftarrow}\cdots\label{sequence: verifying peripheral pi1-stability}%
\end{equation}
where basepoints are suppressed and bonding homomorphisms are compositions of
maps induced by inclusions and change-of-basepoint isomorphisms. Each of those
inclusions is the top row of a commutative diagram%
\[%
\begin{array}
[c]{ccc}%
\partial_{M}N_{i}^{j}\cup N_{k}^{\prime} & \hookleftarrow & \partial_{M}%
N_{i}^{j}\cup N_{k+1}^{\prime}\\
\downarrow\operatorname*{incl} &  & \downarrow\operatorname*{incl}\\
\partial_{M}N_{i}^{j}\cup\widehat{N_{k}^{\prime}} &  & \partial_{M}N_{i}%
^{j}\cup\widehat{N_{k+1}^{\prime}}\\
\downarrow\ \approx &  & \downarrow\ \approx\\
(F_{i}^{j}\times\left\{  0\right\}  )\cup\left(  F_{k}^{\prime}\times\left[
0,1/k\right]  \right)  & \hookleftarrow & (F_{i}^{j}\times\left\{  0\right\}
)\cup\left(  F_{k+1}^{\prime}\times\left[  0,1/k+1\right]  \right)
\end{array}
\]
where the bottom row is an obvious homotopy equivalence, as are all vertical
maps. It follows that the initial inclusion is a homotopy equivalence as well.
As a result, all bonding homomorphisms in
(\ref{sequence: verifying peripheral pi1-stability}) are isomorphisms, so the
sequence is stable.
\end{proof}

\section{Proof of the Manifold Completion Theorem:
sufficiency\label{Section: Proof of Completion Theorem: sufficiency}}

Throughout this section $\left\{  C_{i}\right\}  _{i=1}^{\infty}$ will denote
a clean compact exhaustion of $M^{m}$ with a corresponding cofinal sequence of
clean \smallskip$0$-neigh\-bor\-hoods of infinity $\left\{  N_{i}\right\}
_{i=1}^{\infty}$, each of which has a finite set of connected components
$\left\{  N_{i}^{j}\right\}  _{j=1}^{k_{i}}$. For each $i$ we will let
$W_{i}=\overline{N_{i}\backslash N_{i+1}}$, a compact clean codimension $0$
submanifold of $M^{m}$. Note that $\partial W_{i}$ may be expressed as
$\operatorname*{Fr}N_{i}\cup(\partial_{M}W_{i}\cup\operatorname*{Fr}N_{i+1})$,
a union of two clean codimension $0$ submanifolds of $\partial W_{i}$
intersecting in a common boundary $\partial\left(  \operatorname*{Fr}%
N_{i}\right)  $. (Figures \ref{Figure2} and \ref{Figure1} contain useful
schematics.) The proof of Theorem \ref{Th: Completion Theorem} will be
accomplished by gradually improving the exhaustion of $M^{m}$ so that
ultimately, conditions (\ref{one})-(\ref{three}) of Lemma
\ref{Lemma: geometric characterization} are all satisfied.

\begin{lemma}
\label{Lemma 9.3} If $M^{m}$ is inward tame and $\sigma_{\infty}(M^{m})$
vanishes, then for each $i$, $\sigma(N_{i})$ and $\sigma(N_{i}\backslash
\partial M^{m})$ are both zero.
\end{lemma}

\begin{proof}
By our discussion in \S \ref{Section: Finiteness obstruction}, if $M^{m}$ is
inward tame and $\sigma_{\infty}(M^{m})=0$, then each $N_{i}$ has finite
homotopy type. Since $N_{i}\hookrightarrow N_{i}\backslash\partial M^{m}$ is a
homotopy equivalence, so does $N_{i}\backslash\partial M^{m}$.
\end{proof}

\begin{proposition}
\label{Prop. 9.4} If $M^{m}$ satisfies Conditions (\ref{Char2})-(\ref{Char3})
of Theorem \ref{Th: Completion Theorem} then the $\left\{  C_{i}\right\}  $
and the corresponding $\left\{  N_{i}\right\}  $ can be chosen so that, for
each $i$,

\begin{enumerate}

\item \label{Char 9.1} $\operatorname*{Fr}N_{i}\hookrightarrow N_{i}$ is a
homotopy equivalence, and

\item \label{Char 9.2} $\partial_{M}W_{i}\cup\operatorname*{Fr}N_{i+1}%
\hookrightarrow N_{i}$ is a homotopy equivalence; therefore,

\item the nice relative cobordisms corresponding to $\left(  W_{i}%
,\operatorname*{Fr}N_{i},\partial_{M}W_{i}\cup\operatorname*{Fr}%
N_{i+1}\right)  $ are h-cobordisms.
\end{enumerate}
\end{proposition}

\begin{proof}
By Lemma \ref{Lemma 9.3} and the definition of peripheral $\pi_{1}$-stability
at infinity, we can begin with a clean compact exhaustion $\left\{
C_{i}\right\}  _{i=1}^{\infty}$of $M^{m}$ and a corresponding sequence of
neighborhoods of infinity $\left\{  N_{i}\right\}  _{i=1}^{\infty}$, each with
a finite set of connected components $\left\{  N_{i}^{j}\right\}
_{j=1}^{k_{i}}$, so that for all $i\geq1$ and $1\leq j\leq k_{i}$,

\begin{enumerate}
\item[i)] $N_{i}^{j}$ is inward tame,

\item[ii)] $N_{i}^{j}$ is $\left(  \partial_{M}N_{i}^{j}\right)  $-connected
and $(\partial_{M}N_{i}^{j})$-$\pi_{1}$-stable at infinity, and

\item[iii)] $\sigma_{\infty}\left(  N_{i}^{j}\right)  =0$.
\end{enumerate}

By Lemmas \ref{Lemma: inward tameness of deleted manifolds},
\ref{Lemma: relA pro-pi1 versus Q-A pro-pi1}, and
\ref{Lemma: absolute inward tameness of deleted manifolds}, this implies that

\begin{enumerate}
\item[i$^{\prime}$)] $N_{i}^{j}\backslash\partial_{M}N_{i}^{j}$ is inward tame,

\item[ii$^{\prime}$)] $N_{i}^{j}\backslash\partial M^{m}$ is $1$-ended and has
stable fundamental group at infinity, and

\item[iii$^{\prime}$)] $\sigma_{\infty}\left(  N_{i}^{j}\backslash\partial
M^{m}\right)  =0$.
\end{enumerate}

These are precisely the hypotheses of Siebenmann's Relativized Main Theorem
(\cite[Th.10.1]{Sie65}), so $N_{i}^{j}\backslash\partial M^{m}$ contains an
open collar neighborhood of infinity $V_{i}^{j}\approx\partial V_{i}^{j}%
\times\lbrack0,\infty)$. Following the proof in \cite{Sie65} (similar to what
is done in \cite[Th.3.2]{Obr83}), this can be done so that $\partial N_{i}%
^{j}\backslash\partial M^{m}$ ($=\operatorname*{int}(\operatorname*{Fr}%
N_{i}^{j})$) and $\partial V_{i}^{j}$ contain clean compact codimension $0$
submanifolds $A_{i}^{j}$ and $B_{i}^{j},$ respectively, so that $(\partial
N_{i}^{j}\backslash\partial M^{m})\backslash\operatorname*{int}A_{i}%
^{j}=\partial V_{i}^{j}\backslash\operatorname*{int}B_{i}^{j}\approx\partial
A_{i}^{j}\times\lbrack0,1)$. See Figure \ref{Figure3}.

\begin{figure}[ptb]
\centering
\includegraphics[
height=2.5in,
width=4in
]{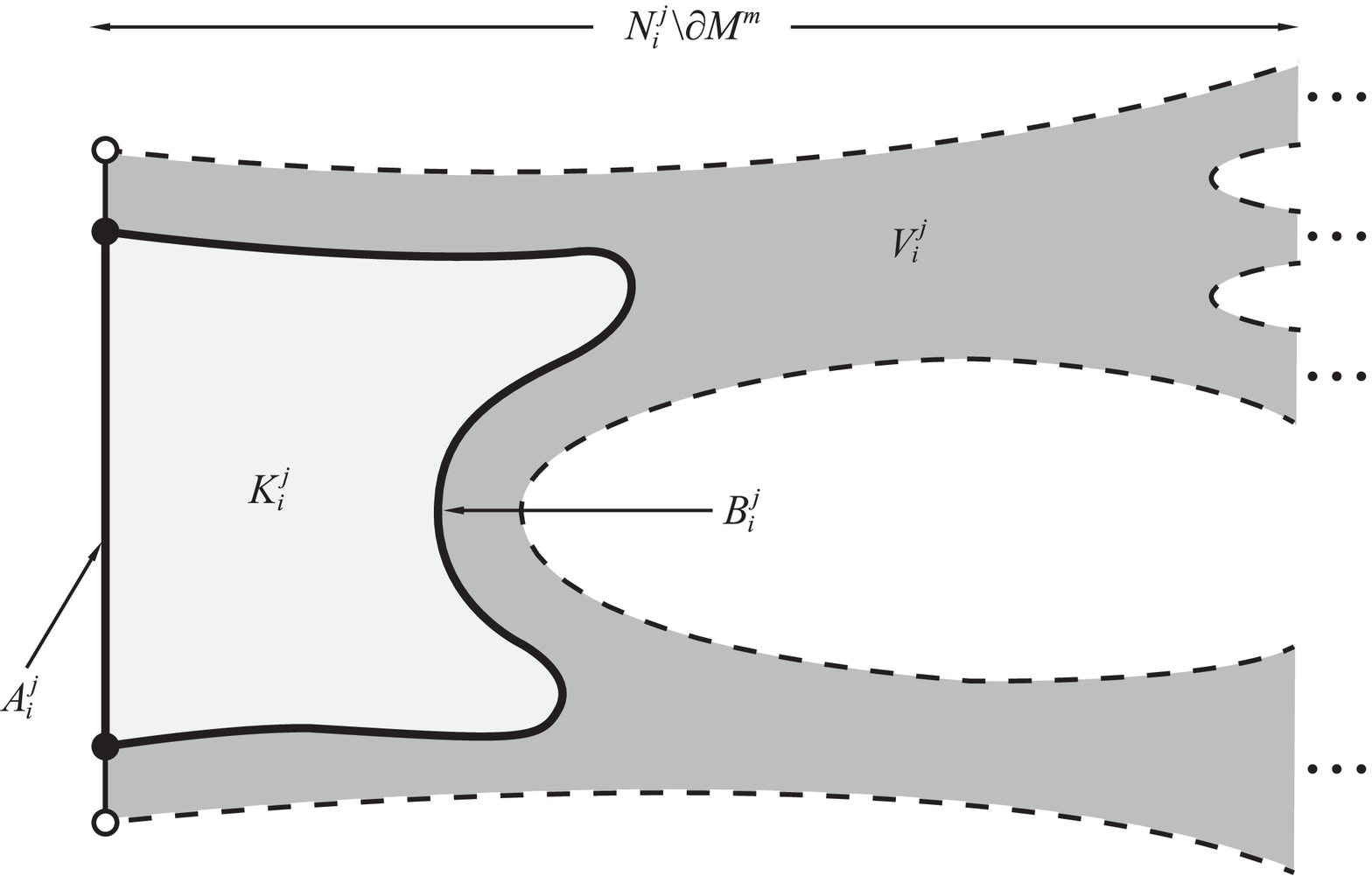}\caption{$V_{i}^{j}\approx\partial V_{i}^{j}\times
\lbrack0,1)$ contained in $N_{i}^{j}\backslash\partial M^{m}$.}%
\label{Figure3}%
\end{figure}

Then $K_{i}^{j}=\overline{N_{i}^{j}\backslash V_{i}^{j}}$ is a clean
codimension $0$ submanifold of $M^{m}$ which intersects $C_{i}$ in $A_{i}^{j}%
$. To save on notation, replace\ $C_{i}$ with $C_{i}\cup\left(  \cup K_{i}%
^{j}\right)  $, which is still a clean compact codimension $0$ submanifold of
$M^{m}$, but with the added property that
\begin{equation}
N_{i}\backslash\partial M^{m}\approx\operatorname*{int}(\operatorname*{Fr}%
N_{i})\times\lbrack0,\infty). \label{Condition: collar condition}%
\end{equation}
Since adding $\partial_{M}N_{i}$ back in does not affect homotopy types, we
also have that
\begin{equation}
\operatorname*{Fr}N_{i}\hookrightarrow N_{i}\text{ is a homotopy equivalence.}
\label{Condition: h.e. condition}%
\end{equation}
Having enlarged the $C_{i}$, pass to a subsequence if necessary to regain the
property that $C_{i}\subseteq\operatorname*{Int}C_{i+1}$ for all $i$.

Letting $N_{i}=\overline{M^{m}\backslash C_{i}}$ gives a nested cofinal
sequence of clean neighborhoods of infinity $\{N_{i}\}$ with the property that
each inclusion $\operatorname*{Fr}N_{i}\hookrightarrow N_{i}$ is a homotopy
equivalence; in other words, we have obtained a pseudo-collar structure on
$M^{m}$. For each $i\geq1$, let $W_{i}=\overline{N_{i}\backslash N_{i+1}}$, a
clean compact codimension $0$ submanifold of $M^{m}$ with $\partial
W_{i}=\operatorname*{Fr}N_{i}\cup(\partial_{M}W_{i}\cup\operatorname*{Fr}%
N_{i+1})$.\medskip

\noindent\textsc{Claim 1. }$\operatorname*{Fr}N_{i}\hookrightarrow W_{i}$
\emph{is a homotopy equivalence.}\medskip

Condition (\ref{Condition: h.e. condition}) applied to $N_{i}$ ensures the
existence a strong deformation retraction $H_{t}$ of $N_{i}$ onto
$\operatorname*{Fr}N_{i}$. That same condition applied to $N_{i+1}$ ensures
the existence of a retraction $r:N_{i+1}\rightarrow\operatorname*{Fr}N_{i+1}$,
which extends to a retraction $\widehat{r}:N_{i}\rightarrow W_{i}$. The
composition $\widehat{r}H_{t}$, restricted to $W_{i}$, gives a deformation
retraction of $W_{i}$ onto $\operatorname*{Fr}N_{i}$.\medskip

\noindent\textsc{Claim 2. }$\partial_{M}W_{i}\cup\operatorname*{Fr}%
N_{i+1}\hookrightarrow W_{i}$ \emph{is a homotopy equivalence.}\medskip

By applying Lemma \ref{Lemma 7.2}, it is enough to show that there exists a
homotopy $H:W_{i}\times\left[  0,1\right]  \rightarrow W_{i}$, fixed on
$\partial(\operatorname*{Fr}N_{i})$, with the property that $H_{1}\left(
W_{i}\right)  \subseteq\partial_{M}W_{i}\cup\operatorname*{Fr}N_{i+1}$. Toward
that end, let $B$ be a collar neighborhood of $\partial_{M}W_{i}$ in $W_{i}$
and let $D=\overline{W_{i}\backslash B}$. Use the collar structure on
$N_{i}\backslash\partial M^{m}$ to obtain a homotopy $K:N_{i}\times\left[
0,1\right]  \rightarrow N_{i}$, fixed on $\partial(\operatorname*{Fr}N_{i})$,
which pushes $N_{i}$ into the complement of $D$; in other words $K_{1}\left(
N_{i}\right)  \subseteq B\cup N_{i+1}$. Compose this homotopy with the
retraction $\widehat{r}:N_{i}\rightarrow W_{i}$ used in the previous claim to
get a homotopy $\widehat{r}K_{t}$ of $W_{i}$ (still fixed on $\partial
(\operatorname*{Fr}N_{i})$) with $\widehat{r}K_{1}\left(  W_{i}\right)
\subseteq B\cup\operatorname*{Fr}N_{i+1}$. Follow this with a homotopy that
deformation retracts $B$ onto $\partial_{M}W_{i}$ while sending
$\operatorname*{Fr}N_{i+1}$ into itself to complete the desired homotopy and
prove Claim 2.\medskip

We can now write $M^{m}=C_{1}\cup W_{1}\cup W_{2}\cup W_{3}\cup\cdots$ where,
for each $i$,

\begin{itemize}
\item $W_{i}$ is a compact clean codimension $0$ submanifold of $M^{m}$,

\item $\partial W_{i}=\operatorname*{Fr}N_{i}\cup(\partial_{M}W_{i}%
\cup\operatorname*{Fr}N_{i+1})$, and

\item both $\operatorname*{Fr}N_{i}\hookrightarrow W_{i}$ and $\partial
_{M}W_{i}\cup\operatorname*{Fr}N_{i+1}\hookrightarrow W_{i}$ are homotopy equivalences.
\end{itemize}

\noindent As such, the corresponding nice relative cobordisms (as described in
Remark \ref{Remark: adjusting to nice relative cobordisms}) are h-cobordisms.
\end{proof}

\begin{proposition}
If $M^{m}$ satisfies Conditions (\ref{Char1})-(\ref{Char4}) of Theorem
\ref{Th: Completion Theorem} the conclusion of Proposition \ref{Prop. 9.4} can
be improved so that, for each $i$, the nice relative cobordisms corresponding
to $\left(  W_{i},\operatorname*{Fr}N_{i},\partial_{M}W_{i}\cup
\operatorname*{Fr}N_{i+1}\right)  $ are s-cobordisms. In that case, $\left(
W_{i},\operatorname*{Fr}N_{i}\right)  \approx\left(  \operatorname*{Fr}%
N_{i}\times\left[  0,1\right]  ,\operatorname*{Fr}N_{i}\times\left\{
0\right\}  \right)  $ for all $i$, and $M^{m}$ is completable.
\end{proposition}

\begin{proof}
By the triviality of $\tau_{\infty}\left(  M^{m}\right)  $, it is possible to
adjust the choices of the $N_{i}$ so that each inclusion $\operatorname*{Fr}%
N_{i}\hookrightarrow W_{i}$ has trivial Whitehead torsion, i.e., $\tau\left(
W_{i},\operatorname*{Fr}N_{i}\right)  =0$, and hence is a simple homotopy
equivalence. As was discussed in \S \ref{Section: Whitehead obstruction}, the
adjustment involves \textquotedblleft lending and borrowing torsion to and
from immediate neighbors of the $W_{i}$\textquotedblright\ as described in
\cite[\S 6]{ChSi76}, except that a Splitting Theorem for finite-dimensional
manifolds (see \cite[p.318]{Obr83}) replaces \cite[Lemma 6.1]{ChSi76}.

To complete the proof, apply the Relative s-cobordism Theorem to each $W_{i}$
then apply Lemma \ref{Lemma: geometric characterization}.
\end{proof}

\section{$\mathcal{Z}$-compactifications and the proof of Theorem
\ref{Th: Stabilization}\label{Section: Stabilization Theorem}}

In this section we prove Theorem \ref{Th: Stabilization}. Since $M^{m}%
\times\left[  0,1\right]  $ satisfies Conditions (\ref{Char2}), (\ref{Char3})
and (\ref{Char4}) of Theorem \ref{Th: Completion Theorem} if and only if
$M^{m}$ satisfies those same conditions (see \cite{ChSi76}), it suffices to
prove the following proposition which is based on work found in \cite{Gui07b}.

\begin{proposition}
\label{Prop: peripheral pi1-stability for MxR}If a manifold $M^{m}$ is inward
tame at infinity, then $M^{m}\times\left[  0,1\right]  $ is peripherally
$\pi_{1}$-stable at infinity.
\end{proposition}

\begin{proof}
Apply Corollary \ref{Corollary: inward tame implies periph. loc. connected} to
obtain a cofinal sequence $\{N_{i}\}$ of clean neighborhoods of infinity for
$M^{m}$ with the property that, for all $i$, each component $N_{i}^{j}$ of
$N_{i}$ is $\partial_{M}N_{i}^{j}$-connected at infinity. Since $\left\{
N_{i}\times\left[  0,1\right]  \right\}  $ is a cofinal sequence of clean
neighborhoods of infinity for $M^{m}\times\lbrack0,1]$ it suffices to show
that the corresponding connected components, $N_{i}^{j}\times\left[
0,1\right]  $, are all $\partial_{M\times\left[  0,1\right]  }(N_{i}^{j}%
\times\left[  0,1\right]  )$-connected and $(\partial_{M\times\left[
0,1\right]  }(N_{i}^{j}\times\left[  0,1\right]  ))$-$\pi_{1}$-stable at
infinity. By Lemmas \ref{Lemma: relA connected versus Q-A 1-ended} and
\ref{Lemma: relA pro-pi1 versus Q-A pro-pi1}, that is equivalent to showing
that, for each $N_{i}^{j}$, $\operatorname*{int}_{M}(N_{i}^{j})\times\left(
0,1\right)  $ is $1$-ended and has stable $\operatorname*{pro}$-$\pi_{1}$ at
that end. Every connected topological space becomes $1$-ended upon crossing
with $\left(  0,1\right)  $, so that condition is immediate. The $\pi_{1}%
$-stability property is proved with a small variation on the main technical
argument from \cite{Gui07b}; in particular, Corollary 3.6 from that paper. The
\textquotedblleft small variation\textquotedblright\ is necessary because the
earlier argument assumed the product of an \emph{open} manifold with $\left(
0,1\right)  $. That issue is easily overcome by arranging that the analog of
homotopy $K_{t}$ used in \cite[Prop.3.3]{Gui07b} sends the manifold interior
of $\operatorname*{Int}_{M}(N_{i}^{j})$ into itself and sends
$\operatorname*{Fr}N_{i}^{j}$ into itself for all $t$. That is easily
accomplished since $\operatorname*{Fr}N_{i}^{j}$ has an open collar
neighborhood at infinity.
\end{proof}

\section{A counterexample to a question of
O'Brien\label{Section: Counterexample}}

We now give a negative answer to a question posed by O'Brien \cite[p.308]%
{Obr83}. \medskip

\noindent\textbf{Question. }\emph{(For a $1$-ended manifold }$M^{m}$\emph{
with $1$-ended boundary), let }$\{V_{i}\}$\emph{ be a cofinal sequence of
clean 0-neighborhoods of infinity. If }$\{\pi_{1}(\partial M^{m}\cup
V_{i})\}_{i\geq1}$\emph{ is stable, does it follow that }$M^{m}$\emph{ is
peripherally }$\pi_{1}$\emph{-stable at infinity?}\medskip

The key ingredient in our counterexamples is a collection of contractible open
$n$-manifolds $W^{n}$ (one for each $n\geq3$), constructed by R. Sternfeld in
his dissertation \cite{Ste77} (see also \cite{Gu18b}). Each $W^{n}$ has the
property that it cannot be embedded in any compact $n$-manifold. Although
these $W^{n}$ have finite homotopy type, they are not inward tame, since they
contain arbitrarily small clean connected neighborhoods of infinity with
non-finitely generated fundamental groups. Our counterexamples will be the
$\left(  n+1\right)  $-manifolds $W^{n}\times\lbrack0,1)$. First a general observation.

\begin{proposition}
\label{Prop 11.1} Let $W^{n}$ be a connected open $n$-manifold. If $W^{n}$ has
finite homotopy type, then $W^{n}\times\lbrack0,1)$ is $1$-ended and inward
tame, with $\sigma_{\infty}\left(  W^{n}\times\lbrack0,1)\right)  =0$.
\end{proposition}

\begin{proof}
It suffices to exhibit arbitrarily small connected clean neighborhood of
infinity in $W^{n}$ with finite homotopy type. Let $N\subseteq W^{n}$ be a
clean neighborhood of infinity and $a\in(0,1)$. By choosing $N$ small and $a$
close to $1$, we can obtain arbitrarily small neighborhoods of infinity in
$W^{n}\times\lbrack0,1)$ of the form
\[
V(N,a)=\left(  N\times\lbrack0,1)\right)  \cup\left(  W^{n}\times\lbrack
a,1)\right)  .
\]
Since $V\left(  N,a\right)  $ deformation retracts onto $W^{n}\times\left\{
a\right\}  $, it is connected and has finite homotopy type.
\end{proof}

\begin{example}
Consider the $(n+1)$-manifold $M^{n+1}=W^{n}\times\lbrack0,1)$, where $W^{n}$
is the Sternfeld $n$-manifold ($n\geq3$) described above. Then $\partial
M^{n+1}=W^{n}\times\{0\}$. A standard duality argument shows that \emph{every}
contractible open manifold of dimension $\geq2$ is $1$-ended. Let $\{N_{i}\}$
be a cofinal sequence of clean connected neighborhoods of infinity in $W^{n}$,
and for each $i\geq1$, let $V_{i}=V\left(  N_{i},\frac{i}{i+1}\right)  $, as
defined in the previous proof. By Seifert-van Kampen, each $V_{i}\cup\partial
M^{n+1}$ is simply connected, so the inverse sequence $\{\pi_{1}(\partial
M^{n+1}\cup V_{i})\}_{i\geq1}$ is pro-trivial, hence, stable.

To see that $M^{n+1}$ is not peripherally $\pi_{1}$-stable at infinity, first
assume that $n\geq5$. Then, if $M^{n+1}$ were peripherally $\pi_{1}$-stable at
infinity, it would be completable by Theorem \ref{Th: Completion Theorem}.
(The triviality of $\tau_{\infty}\left(  M^{n+1}\right)  $ is immediate since
$M^{n+1}$ is simply connected at infinity, which follows from the simple
connectivity of the $V_{i}$.) But, if $\widehat{M}^{n+1}$ were a completion,
then $W^{n}\times\left\{  0\right\}  \hookrightarrow$ $\partial\widehat{M}%
^{n+1}$would be an embedding into a closed $n$-manifold, contradicting
Sternfeld's theorem.

To obtain analogous examples when $n=3$ or $n=4$, we cannot rely on the
Manifold Completion Theorem. But a direct analysis of the fundamental group
calculations in Sternfeld's proof reveals that the peripheral
$\operatorname*{pro}$-$\pi_{1}$-systems arising in $W^{n}\times\lbrack0,1)$
are nonstable in those dimensions as well.
\end{example}

\section{Proof of Lemma
\ref{Lemma: equivalence of peripheral pi1-stability conditions}
\label{Section: Equivalence of various peripheral pi1-stability conditions}}

We now return to Lemma
\ref{Lemma: equivalence of peripheral pi1-stability conditions}, which asserts
that the two natural candidates for the definition of \textquotedblleft
peripherally $\pi_{1}$-stable at infinity\textquotedblright\ (the global
versus the local approach)\ are equivalent for inward tame manifolds. The
intuition behind the lemma is fairly simple. If $M^{m}$ contains arbitrarily
small \smallskip$0$-neigh\-bor\-hoods of infinity $N$ with the property that
each component $N^{j}$ is $\partial_{M}N^{j}$-$\pi_{1}$-stable at infinity,
then those components provide arbitrarily small neighborhoods of the ends
satisfying the necessary $\pi_{1}$-stability condition. Conversely, if each
end $\varepsilon$ has arbitrarily small strong \smallskip$0$-neigh\-bor\-hoods
$P$ that are $\partial_{M}P$-$\pi_{1}$-stable at infinity, we can use the
compactness of the set of ends (in the Freudenthal compactification) to find,
within any neighborhood of infinity, a finite collection $\left\{
P_{1},\cdots,P_{k}\right\}  $ of such neighborhoods which cover the end of
$M^{m}$. If we can do this so the $P_{i}$ are pairwise disjoint, we are
finished---just let $N=\cup P_{i}$. That is not as easy as one might hope, but
we are able to attain the desired conclusion by proving the following proposition.

\begin{proposition}
\label{Proposition: final step of peripheral pi1-stability characterization}%
Suppose $M^{m}$ is inward tame and each end $\varepsilon$ has arbitrarily
small strong $0$-neigh\-bor\-hoods $P_{\varepsilon}$ that are $\partial
_{M}P_{\varepsilon}$-$\pi_{1}$-stable at infinity. Then every strong partial
\smallskip$0$-neigh\-bor\-hood of infinity $Q\subseteq M^{m}$ is $\partial
_{M}Q$-$\pi_{1}$-stable at infinity.
\end{proposition}

Our proof requires that we break the stability condition into a pair of weaker
conditions. An inverse sequence of groups is:

\begin{itemize}
\item \emph{semistable} (sometimes called \emph{pro-epimorphic}) if it is
pro-isomorphic to an inverse sequence of surjective homomorphisms;

\item \emph{pro-monomorphic} of it is pro-isomorphic to an inverse sequence of
injective homomorphisms.
\end{itemize}

\noindent It is an elementary fact that an inverse sequence is stable if and
only if it is both semistable and pro-monomorphic.

We will make use of the following topological characterizations of the above
properties, when applied to pro-$\pi_{1}$. In these theorems, a
\textquotedblleft space\textquotedblright\ should be locally compact, locally
connected, and metrizable.

\begin{proposition}
\label{Proposition: topological characterizations of pro-mono and pro-epi}Let
$X$ be a $1$-ended space and $r:[0,\infty)\rightarrow X$ a proper ray. Then
pro-$\pi_{1}\left(  X,r\right)  $ is

\begin{enumerate}
\item semistable if and only if, for every compact set $C\subseteq X$, there
exists a larger compact set $D\subseteq X$ such that for any compact set $E$
with $D\subseteq E\subseteq X$, every loop in $X\backslash D$ with base point
on $r$ can be pushed into $X\backslash E$ by a homotopy with image in
$X\backslash C$ keeping the base point on $r$, and

\item pro-monomorphic if and only if $X$ contains a compact set $C$ with the
property that, for every compact set $D$ with $C\subseteq D\subseteq X$, there
exists a compact set $E\supseteq D$ with the property that every loop in
$X\backslash E$ that contracts in $X\backslash C$ also contracts in
$X\backslash D$.
\end{enumerate}
\end{proposition}

These are standard. See, for example \cite{Geo08} or \cite{Gui16}. In the case
that pro-$\pi_{1}(X,r)$ is pro-monomorphic, the compact set $C$ in the above
proposition is called a $\pi_{1}$\emph{-core }for $X$. Notice that, by
Proposition
\ref{Proposition: topological characterizations of pro-mono and pro-epi}, the
property of ($1$-ended) $X$ having pro-monomorphic pro-$\pi_{1}(X,r)$ is
independent of the choice of $r$.

It is a non-obvious (but standard) fact that having semistable pro-$\pi
_{1}(X,r)$ is also independent of the choice of $r$. As for the
characterization of semistable pro-$\pi_{1}(X,r)$, we are mostly interested in
the following easy corollary.

\begin{corollary}
\label{Corollary: arc pushing property}If $X$ is a $1$-ended space and
pro-$\pi_{1}\left(  X,r\right)  $ is semistable for some (hence every) proper
ray $r$, then for each compact set $C\subseteq X$, there is a larger compact
set $D\subseteq X$ such that, for every compact set $E\subseteq X$ and every
path $\lambda:\left[  0,1\right]  \rightarrow X\backslash D$ with
$\lambda\left(  \left\{  0,1\right\}  \right)  \subseteq E$, there is a path
homotopy in $X\backslash C$ taking $\lambda$ to a path $\lambda^{\prime}$ in
$X\backslash E.$
\end{corollary}

We are now ready for our primary task.

\begin{proof}
[Proof of Proposition
\ref{Proposition: final step of peripheral pi1-stability characterization}]Let
$Q$ be a strong partial \smallskip$0$-neigh\-bor\-hood of infinity in $M^{m}$.
By Lemma \ref{Lemma: relA pro-pi1 versus Q-A pro-pi1}, proving that $Q$ is
$\partial_{M}Q$-$\pi_{1}$-stable at infinity is equivalent to proving that the
$1$-ended space $Q\backslash\partial M^{m}$ has stable $\operatorname*{pro}%
$-$\pi_{1}$. We will take the latter approach.

By Lemma \ref{Lemma: inward tameness of deleted manifolds} $Q\backslash
\partial M^{m}$ is inward tame, so a modification of the argument in
\cite[Prop. 3.2]{GuTi03} ensures that $\operatorname*{pro}$-$\pi_{1}\left(
Q\backslash\partial M^{m},r\right)  $ is semistable. It is therefore enough to
show that $\operatorname*{pro}$-$\pi_{1}\left(  Q\backslash\partial
M^{m},r\right)  $ is pro-monomorphic. We will do that by verifying the
condition described in Proposition
\ref{Proposition: topological characterizations of pro-mono and pro-epi},
i.e., we will show that $Q\backslash\partial M^{m}$ contains a $\pi_{1}$-core.

By hypothesis, each end $\varepsilon$ of $Q$ has a strong \smallskip
$0$-neigh\-bor\-hood $P_{\varepsilon}$ which is $\partial_{M}P_{\varepsilon}%
$-$\pi_{1}$-stable at infinity and lies in $\operatorname*{Int}_{M}Q$. Since
the set of ends of $Q$ is compact in the Freudenthal compactification, there
is a finite subcollection $\{P_{\varepsilon_{i}}\}_{i=1}^{k}$ whose union is a
neighborhood of infinity in $Q$. Place the collection of submanifolds
$\left\{  P_{\varepsilon_{i}}\right\}  _{i=1}^{k}$ in general
position.\medskip

\noindent\textsc{Claim 1. }\emph{For each }$\Omega\subseteq\left\{
1,\cdots,k\right\}  $\emph{ the set }$\cap_{j\in\Omega}P_{\varepsilon_{j}}%
$\emph{ has finitely many components, each of which is a clean codimension
}$0$\emph{ submanifold of }$M^{m}$\emph{.\medskip}

General position ensures that each component is a clean codimension $0$
submanifold of $M^{m}$. Since each $P_{\varepsilon_{j}}$ is a closed subset of
$M^{m}$ each component $T$ of $\cap_{j\in\Omega}P_{\varepsilon_{j}}$ is closed
in $M^{m}$, and since $T$ cannot also be open in $M^{m}$ it must have nonempty
frontier. Since $\left\{  P_{\varepsilon_{j}}\right\}  _{j\in\Omega}$ is in
general position, so also is the collection of (compact) frontiers, $\left\{
\operatorname*{Fr}P_{\varepsilon_{j}}\right\}  _{j\in\Omega}$. So, for each
$i\neq j$ in $\Omega$, $\Delta_{i,j}=\operatorname*{Fr}P_{\varepsilon_{i}}%
\cap\operatorname*{Fr}P_{\varepsilon_{j}}$ is a clean codimension $1$
submanifold of $\operatorname*{Fr}P_{\varepsilon_{i}}$ and $\operatorname*{Fr}%
P_{\varepsilon_{j}}$. The union of these $\Delta_{i,j}$ separate $\cup
_{j=1}^{k}\operatorname*{Fr}P_{\varepsilon_{j}}$ into finitely many pieces,
and since the frontier of each $T$ is a union of these pieces, there can only
be finitely many such $T$.

Choose an embedding $b:\partial M^{m}\times\lbrack0,1]\rightarrow M^{m}$ with
$b\left(  x,0\right)  =x$ for all $x\in\partial M^{m}$ and whose image $B$ is
a regular neighborhood of $\partial M^{m}$ in $M^{m}$. With some additional
care, arrange that $B$ intersects: $Q$ in $b\left(  \partial_{M}Q\times\left[
0,1\right]  \right)  $; each $P_{\varepsilon_{i}}$ in $b\left(  \partial
_{M}P_{\varepsilon_{i}}\times\left[  0,1\right]  \right)  $; and (more
specifically) each component $T$ of each finite intersection $\cap_{j\in
\Omega}P_{\varepsilon_{j}}$ in $b\left(  \partial_{M}T\times\left[
0,1\right]  \right)  $. For each $0\leq s<t\leq1$, let $B^{[s,t]}=b\left(
\partial M^{m}\times\lbrack s,t]\right)  $, $B^{(s,t)}=b\left(  \partial
M^{m}\times(s,t)\right)  $, etc. For $A\subseteq\partial M^{m}$, let
$B_{A}=b\left(  A\times\lbrack0,1]\right)  $ and define $B_{A}^{[s,t]}$,
$B_{A}^{(s,t)}$, etc. analogously.

By hypothesis and Proposition
\ref{Proposition: topological characterizations of pro-mono and pro-epi} we
can choose a clean codimension $0$ compact $\pi_{1}$-core $C_{i}$ for each
$P_{\varepsilon_{i}}\backslash\partial M^{m}$. Then choose $t$ so small that
$B^{[0,t]}\cap(\cup_{i=1}^{k}C_{i})=\varnothing$. Let $C_{0}^{\prime}%
\equiv\overline{Q\backslash\cup_{i=1}^{k}P_{\varepsilon_{i}}}$, then let
$C_{0}=C_{0}^{\prime}\backslash B^{[0,t)}$ so that $C_{0}$ is a compact clean
codimension $0$ submanifold of $Q\backslash\partial M^{m}$. Let $C=\cup
_{i=0}^{k}C_{i}\subseteq Q\backslash\partial M^{m}$. Notice that the
collection $\left\{  B_{\partial_{M}Q}^{[0,t]},P_{\varepsilon_{1}}%
,\cdots,P_{\varepsilon_{k}}\right\}  $ covers $Q\backslash\operatorname*{Int}%
_{Q}C$.

Choose a clean codimension $0$ compact submanifold of $D^{\prime}\subseteq
Q\backslash\partial M^{m}$ so large that

\begin{itemize}
\item[i)] $\operatorname*{Int}_{Q}D^{\prime}\supseteq C$,

\item[ii)] $D^{\prime}$ contains every compact component of $\cap_{j\in\Omega
}P_{\varepsilon_{j}}$ for all $\Omega\subseteq\left\{  1,\cdots,k\right\}  $, and

\item[iii)] for any compact set $E\subseteq Q\backslash\partial M^{m}$ such
that $D^{\prime}\subseteq E$, if $\lambda$ is a path in $T\backslash\partial
M^{m}$, where $T$ is an unbounded component of $P_{\varepsilon_{i}}\cap
P_{\varepsilon_{j}}$ for some $i,j\in\left\{  1,\cdots,k\right\}  $, and
$\lambda$ lies outside $D^{\prime}$ with endpoints outside $E$, then there is
a path homotopy of $\lambda$ in $(T\backslash\partial M^{m})\backslash C$
pushing $\lambda$ outside $E$. (This uses Corollary
\ref{Corollary: arc pushing property} and the fact that each $T$, being a
clean partial neighborhood of infinity in $M^{m}$, has the property that
$T\backslash\partial M^{m}$ has finitely many ends, each with semistable
$\operatorname*{pro}$-$\pi_{1}$.)
\end{itemize}

Now choose a compact set $D\subseteq Q\backslash\partial M^{m}$ such that

\begin{itemize}
\item[i$^{\prime}$)] $D\supseteq D^{\prime}$,

\item[ii$^{\prime}$)] for every $\Omega\subseteq\left\{  1,\cdots,k\right\}  $
and every unbounded component $T$ of $\cap_{j\in\Omega}P_{\varepsilon_{j}}$,
each $x\in(T\backslash\partial M^{m})\backslash D$ can be pushed to infinity
in $(T\backslash\partial M^{m})\backslash D^{\prime}$. (This is possible since
there are only finitely such $T$.)

\item[iii$^{\prime}$)] \label{Property iii')}if $x=b\left(  y,t_{0}\right)
\in B\backslash D$, then $b\left(  y\times\left[  0,t_{0}\right]  \right)
\cap D^{\prime}=\varnothing.\medskip$
\end{itemize}

\noindent\textsc{Claim 2. }$D$\emph{ is a }$\pi_{1}$\emph{-core for
}$Q\backslash\partial M^{m}$\emph{.}$\medskip$

Toward that end, let $F$ be a compact subset of $Q\backslash\partial M^{m}$
containing $D$, then choose $G\subseteq Q\backslash\partial M^{m}$ to be an
even larger compact set with the following property:

\begin{itemize}
\item[(\dag)] for each $i\in\left\{  1,\cdots,k\right\}  $, loops in
$P_{\varepsilon_{i}}\backslash\partial M^{m}$ lying outside $G$ which contract
in $(P_{\varepsilon_{i}}\backslash\partial M^{m})\backslash C$, also contract
in $(P_{\varepsilon_{i}}\backslash\partial M^{m})\backslash F$.
\end{itemize}

Let $\alpha:\left[  0,1\right]  \times\left[  0,1\right]  \rightarrow\left(
Q\backslash\partial M^{m}\right)  \backslash D$. The interiors of sets
$\left\{  B_{\partial_{M}Q}^{[0,t]},P_{\varepsilon_{1}},\cdots,P_{\varepsilon
_{k}}\right\}  $ cover $\left(  Q\backslash\partial M^{m}\right)  \backslash
D$, so we can subdivide $\left[  0,1\right]  ^{2}$ into subsquares $\left\{
R_{t}\right\}  $ so small that the image of each $R_{t}$ lies in $B^{(0,t)}$
or one of the $P_{\varepsilon_{i}}\backslash\partial M^{m}$ and hence, in
$B^{(0,t)}\backslash D$ or one of the $(P_{\varepsilon_{i}}\backslash\partial
M^{m})\backslash D$. Since each vertex of this subdivision is sent to a point
$x$ in $B^{(0,t)}\backslash D$ and/or $T\backslash D$, where $T$ is an
unbounded component of the intersection of the $P_{\varepsilon_{i}}$ which
contain the images of the subsquares containing that vertex, then by the
choice of $D$ we can push $x$ into $(Q\backslash\partial M^{m})\backslash G$
along a path that does not leave $T$ and does not intersect $D^{\prime}$. In
those cases where $x=b\left(  y,t_{0}\right)  \in B^{(0,t)}\backslash D$, push
$x$ out of $G$ along $b\left(  y\times\left(  0,1\right)  \right)  $, so that
the track also stays in $B^{(0,t)}\backslash D^{\prime}$, by property
(iii$^{\prime}$).

Doing the above for each vertex adjusts $\alpha$ up to homotopy in
$(Q\backslash\partial M^{m})\backslash D^{\prime}$ so that each vertex of the
subdivision is taken into $(Q\backslash\partial M^{m})\backslash G$ and each
$R_{t}$ is still taken into the same $P_{\varepsilon_{i}}$ (or $B^{(0,t)}$) as before.

Next we move to the 1-skeleton of our subdivision of $\left[  0,1\right]
^{2}$. If an edge $e$ is the intersection $R_{t}\cap R_{t^{\prime}}$ of two
squares, i.e., $e$ is not in $\partial(\left[  0,1\right]  ^{2})$, we use
property (iii) to adjust $\alpha$ up to homotopy so $e$ is mapped into
$(Q\backslash\partial M^{m})\backslash G$, noting that this homotopy may
causes the \textquotedblleft new\textquotedblright\ $\alpha$ to drift into
$(Q\backslash\partial M^{m})\backslash C$. (If $e$ is sent into $B^{(0,t)}$,
we can use (iii$^{\prime}$) to ensure that the push stays in $B^{(0,t)}%
\backslash D^{\prime}$ as well.)

Do the above for each edge until the entire 1-skeleton of the subdivision of
$\left[  0,1\right]  ^{2}$ is mapped into $(Q\backslash\partial M^{m}%
)\backslash G$. The image of $\alpha$ now lies in $(Q\backslash\partial
M^{m})\backslash C$. Notice that the restriction of $\alpha$ to each $R_{t}$
is a map of a disk into a single $P_{\varepsilon_{i}}$ (or $B^{(0,t)}$)
missing $C_{i}$ with boundary being mapped into $P_{\varepsilon_{i}}\backslash
G$. So by the choice of $G$, we may redefine $\alpha$ on $R_{t}$ to be the
same on its boundary, but to take $R_{t}$ into $P_{\varepsilon_{i}}\backslash
F$ or $B^{(0,t)}\backslash F$. Assembling the\ $\left.  \alpha\right\vert
_{R_{t}}$ we get a map $\alpha^{\prime}:\left[  0,1\right]  \times\left[
0,1\right]  \rightarrow\left(  Q\backslash\partial M^{m}\right)  \backslash F$
that agrees with $\alpha$ on $\partial(\left[  0,1\right]  ^{2})$.
\end{proof}

\bibliographystyle{amsalpha}
\bibliography{Biblio}
{}

\end{document}